\documentclass[12pt]{amsart}
\usepackage{comment}

\title{The Method Of Thue-Siegel For Binary Quartic Forms}
\author{Shabnam Akhtari}
\address{ Max-Planck-Institut f\"{u}r Mathematik \newline
 Vivatsgasse 7\newline
53111 Bonn Germany
}
\email{akhtari@mpim-bonn.mpg.de}
 

\begin{document}
\baselineskip=17pt
\maketitle

\newtheorem{thm}{Theorem}[section]
\newtheorem{prop}[thm]{Proposition}
\newtheorem{lemma}[thm]{Lemma}
\newtheorem{cor}[thm]{Corollary}
\newtheorem{conj}[thm]{Conjecture}

\begin{abstract}
We will use Thue-Siegel method, based on  Pad\'e approximation via
hypergeometric functions, to give  upper bounds for the number of integral solutions to the equation
$|F(x , y)| = 1$ as well as the inequalities $\left|F(x , y)\right| \leq  h$, 
for a certain family of irreducible quartic binary forms. 
  \end{abstract}

\section{Introduction}

In 1909, Thue \cite{Thu2} proved that if $F(x , y)$ is an irreducible binary form of degree at least $3$ with integer coefficients, and $h$ a nonzero integer, then the equation $F(x , y) = h$ has only finitely many solutions  in integers $x$ and $y$ .

In this paper we will consider  irreducible  binary quartic forms  with integer coefficients, i.e.  polynomials of the shape
$$
F(x , y) = a_{0}x^{4} + a_{1}x^{3}y + a_{2}x^{2}y^{2} + a_{3}xy^{3} + a_{4}y^{4}. 
$$
The discriminant $D$ of $F$ is given by
$$
D = D_{F} = a_{0}^{6} (\alpha_{1} - \alpha_{2})^{2}  (\alpha_{1} - \alpha_{3})^{2}   (\alpha_{1} - \alpha_{4})^{2}   (\alpha_{2} - \alpha_{3})^{2}  (\alpha_{2} - \alpha_{4})^{2}  (\alpha_{3} - \alpha_{4})^{2} ,
$$
where $\alpha_{1}$ , $\alpha_{2}$, $\alpha_{3}$ and $\alpha_{4}$ are the roots of 
$$
F(x , 1) =  a_{0}x^{4} + a_{1}x^{3} + a_{2}x^{2} + a_{3}x + a_{4} . 
$$
Here, we will recall some well-known fact about the invariants of quartic forms. We refer the reader to \cite{Cre2} for more details. The invariants of $F$ form a ring, generated by two invariants of weights $4$ and $6$, namely
$$
I = I_{F} = a_{2}^{2} - 3a_{1}a_{3} + 12a_{0}a_{4} 
$$
and
$$
J = J_{F} = 2a_{2}^{3} - 9a_{1}a_{2}a_{3} + 27 a_{1}^{2}a_{4} - 72 a_{0}a_{2}a_{4} + 27a_{0}a_{3}^{2}.
$$
These are algebraically independent and every invariant is a polynomial in $I$ and $J$. For the invariant $D$, we have
$$
27D = 4I^3 - J^2.
$$
In what follows, we will just consider the forms $F$ for which the quantity $J_{F}$ is $0$; i.e. for which we have
$$
27D = 4I^{3}.
$$

Let $h$ be a positive integer. The number of solutions in integers $x$ and $y$ of the equation
\begin{equation}\label{12}
\left|F(x , y)\right| = h.
\end{equation}
 will be the focus of our  study in this paper. In Section \ref{TS2}, we will show that to apply a classical theorem of Thue \cite{Thu2} from Diophantine approximation  to a quartic form $F$, one needs to assume  $J_{F} = 0$.

   \begin{thm}\label{main2}
 Let $F(x , y)$ be an irreducible binary quartic form with integer coefficients and positive discriminant that splits in $\mathbb{R}$. If $J_{F} = 0$, then the Diophantine equation $\left|F(x , y)\right| = 1$ possesses at most $12$ solutions in integers $x$ and $y$ (with $(x , y)$ and $(-x , -y)$ regarded as the same) .
 \end{thm}
 
In Section \ref{FWSD2}, we will summarize the result of our computations for binary forms with small discriminant. We will give some examples for quartic binary forms $F(x , y)$ satisfying the hypotheses of Theorem \ref{main2}, where $\left|F(x , y)\right| = 1$ has $4$ or $3$ solutions in inegers $x$ and $y$. The author is not aware of any quartic binary form $F$ for which $\left|F(x , y)\right| = 1$ has more than $4$ solutions.

In \cite{Akho2} different methods  are used to give an upper bound $61$ upon the number of integral solutions to the equation $\left|F(x , y)\right| = 1$, where $F$ is  an irreducible binary quartic form with no restriction on the value of $J_{F}$ and with $\left|D_{F}\right|$ large enough. Moreover, it is shown in \cite{Akho2} that if the irreducible binary quartic form $F$ splits in $\mathbb{R}$ and have large discriminant, the Diophantine equation 
$\left|F(x , y)\right| = 1$ has at most $36$ solutions in integers $x$ and $y$.
  
\begin{thm}\label{main22}
Let $F(x , y)$ be a reduced irreducible binary quartic form with integer coefficients and positive discriminant that splits in $\mathbb{R}$. If $J_{F} = 0$, then the inequality  $\left|F(x , y)\right| \leq  h$ possesses at most $12$ co-prime solutions $(x , y)$, with  $ |y| \geq \frac{h^{3/4}}{ (3I)^{1/8}}$.
  \end{thm}
 The definition of a reduced form is given in  Section \ref{EF2}. It turns out that each quartic binary form is equivalent to a reduced one (see \cite{Cre2}).

One reason for us to be interested in these results, despite what are apparently quite serious restriction upon $F$, is that  we know important families of quartic forms with these properties. For example a solution to the equation $aX^4 - bY^2 = 1$ gives rise  to a solution to the Thue equation
$$
x^4 + 4tx^3y - 6tx^2 y^2 - 4t^2xy^3 + t^2y^4 = t_{1}^2,
$$ 
where $t_{1} | t$. We have applied the methods of this paper to treat the above Thue equation in \cite{Akh2}.
 
  The method of Thue and Siegel based on Pad\'e approximation to binomial functions applies to broad families of binomial Thue equations, and both so-called "quantitative" results (see the works of Evertse \cite{Eve2, Eve3}, for example) as well as effective results (via effective irrationality measure from Baker \cite{Bak2, Bak12} onwards) can be obtained from it. This method
has also been used to study binary cubic forms with positive discriminant, for decades (see \cite{Ev2}, \cite{Ben2}). In $1939$, Krechmar \cite{Kre2} showed that when the discriminant of quartic form $F(x , y)$ is sufficiently large ( $D_{F} \gg h^{216/5}$),  the equation (\ref{12}) has at most $20$ solutions in integers $x$ and $y$, provided that  $J_{F} = 0$ and all roots of $F(x , 1)$ are real numbers.  We will use a refinement of Thue-Siegel method by Evertse \cite{Ev2} to obtain our results.

 \section{The Method Of Thue-Siegel}\label{TS2}
 The main purpose of this section is to explain why we need the restriction $J_{F} =0$ in the statements of our Theorems. The answer is hidden in the method we use, the method of Thue-Siegel. The relationship between a system of approximations to an arbitrary cubic irrationality and Pad\'e  approximations to $\sqrt[3]{1 - x}$ was first established by Thue \cite{thu352}.  Siegel  \cite{Sie282, Sie292} identified approximating polynomials in  Thue's papers \cite{thu352, thu382} with hypergeometric polynomials and applied this method to  bounding  the number of solutions to Diophantine equation $f(x , y) = k$, for certain binary forms $f(x , y)$ of degree $r$. He also established bounds for the number of solutions to 
 $$
 ax^{n} - by^{n} = c , 
$$
 where $n \geq 3$ \cite{Sie302}.

\begin{lemma}\label{thue2}
 Suppose that $P(x)$ is a polynomial of degree $n$ and there is  a quadratic polynomial $U(x)$ such that 
  \begin{equation}\label{ch2}
 U(x)P''(x) - (n -1)U'(x) P'(x) + \frac{n (n - 1)}{2} U''(x) P(x) = 0 .
 \end{equation}
  Let 
  $$
  Y(x) = 2 U(x) P'(x) - n U'(x)P(x)
  $$
 and 
 $$
 h = \frac{n^{2} - 1}{4} \big{(}U'(x)^{2} - 2U(x) U''(x) \big{)}. 
 $$
  Consider the recurrences 
 $$ 
 P_{r+1}(x) = k_{r}Y(x)P_{r}(x) - P(x)^{2} P_{r-1}(x) ;
 $$
 $$ 
 Q_{r+1}(x) = k_{r}Y(x)Q_{r}(x) - P(x)^{2} Q_{r-1}(x) ,
 $$
  with the initial conditions
 $$
 P_{0}(x) =  Q_{0}(x) = \frac{2}{3} h , 
 $$
 $$
 P_{1}(x) = U(x) P'(x) - \frac{n -1}{2} U'(x) P(x) ,
 $$
 $$
 Q_{1}(x) = x P_{1}(x) - U(x) P(x) ,
 $$
 where
  $$
  c_{1} = \frac{3}{2} , \        c_{2} = \frac{2(2n - 1)(2n +1)}{3 (n - 1) (n + 1)} h  , \         k_{r}c_{r} = \frac{2r + 1}{2} 
  $$
 and 
 $$
 \frac{c_{r+1} - c_{r-1}}{k_{r}} = 2h \frac{n^{2}}{(n -1) (n + 1)} .
 $$
 Then  polynomials $P_{r}(x)$, $Q_{r}(x)$ are of degree $rn + 1$ and satisfy equation 
  $$
  \alpha P_{r}(x) - Q_{r}(x) = (x - \alpha )^{2r + 1} R_{r}(x) 
  $$
  for a polynomial $R_{r}(x)$ .
  \end{lemma}

  To apply Theorem \ref{thue2}  to the polynomial $P(x) =  a_{0}x^{4} + a_{1}x^{3} + a_{2}x^{2} + a_{3}x + a_{4}$, suppose that for a quadratic polynomial $U(x) = u_{2}x^{2} + u_{1}x + u_{0}$, we have
 
 \begin{eqnarray*}
 0 &=& U(x)P''(x) - 3U'(x) P'(x) + 6 U''(x) P(x)  \\  \nonumber
  &=& (12a_{0}u_{0}-3a_{1}u_{1}+2a_{2}u_{2}) x^{2} + (6a_{1}u_{0}-4a_{2}u_{1}+6a_{3}u_{2}) x \\ \nonumber
& & + 2a_{2}u_{0}-3a_{3}u_{1}+12a_{4}u_{2}.
  \end{eqnarray*}
  This implies that 
 \begin{displaymath}
  \left( \begin{array}{ccc}
  12a_{0} & -3a_{1} &  2a_{2} \\
   3a_{1}  & -2a_{2} &  3a_{3} \\
    2a_{2} &  -3a_{3}&  12a_{4} 
  \end{array} \right) 
  \left( \begin{array}{c}
  u_{0} \\
  u_{1}  \\
  u_{2}
  \end{array} \right)= 0 .
\end{displaymath}
  Therefore,
 \begin{eqnarray*}
& & \textrm{det}  \left( \begin{array}{ccc}
  12a_{0} & -3a_{1} &  2a_{2} \\
   3a_{1}  & -2a_{2} &  3a_{3} \\
    2a_{2} &  -3a_{3}&  12a_{4}  
  \end{array} \right)\\
  & = & 4(2a_{2}^{3} - 9a_{1}a_{2}a_{3} + 27 a_{1}^{2}a_{4} - 72 a_{0}a_{2}a_{4} + 27a_{0}a_{3}^{2})\\ \nonumber
& = &4J = 0 .
 \end{eqnarray*}

 In this paper, we always suppose that $J = 0$. In Section \ref{resolvent2}, we will show  that if $J_{F} = 0$ then there are linear forms $\xi = \xi(x , y)$ and $\eta = \eta(x , y)$ so that 
$$F(x , y) = \frac{1}{8\sqrt{3IA_{4}}}\left(\xi^{4} - \eta^{4}\right),$$
where the quantity $A_{4}$ is defined in (\ref{Ai2}).
 We will use Pad\'e approximation via hypergeometric polynomials to approximate ${\eta}/{\xi}$ with rational integers.
The main idea here is to replace the construction of a family of dense approximations to ${\eta}/{\xi}$, by a family of rational approximations to the function $(1 - z)^{1/4}$. Consider the system of linear forms $R_{r}(z) = -Q_{r}(z) + (1-z)^{1/4}P_{r}(z)$ that approximate $(1-z)^{1/4}$ at $z = 0$,  such that $R_{r}(z) = z^{2r+1}\bar{R}_{r}(z)$, $\bar{R}_{r}(z)$ is regular at $z=0$, and $P_{r}(z)$ and $Q_{r}(z)$ are polynomials of degree $r$. Thue \cite{thu342, thu352} explicitly found polynomials $P_{r}(z)$ and $Q_{r}(z)$ 
 and  Siegel \cite{Sie282} identified them in terms of hypergeometric  polynomials. Refining the method of Siegel,  Evertse \cite{Ev2} used the theory of hypergeometric functions to give an upper bound for the number of solutions to the equation $f(x , y) = 1$, where $f$ is a cubic binary form with positive discriminant. Here we adjust Lemma $4$ of \cite{Ev2} for quartic forms.

\begin{lemma}\label{hyp2}
Let $r$, $g$ be integers with $r \geq 1$, $g \in \{ 0 , 1 \}$. Put
\begin{eqnarray}\label{AB2}\nonumber
A_{r, g} (z) & = &  \sum_{m =0}^{r}{r - g + \frac{1}{4} \choose m} {2r - g - m \choose r - g}   (-z)^{m},   \\  
B_{r, g} (z) & = & \sum_{m =0}^{r-g}{r - \frac{1}{4} \choose m} {2r - g - m \choose r }   (-z)^{m}.
 \end{eqnarray}
  \flushleft
 \begin{itemize}
 \item[(i)] There exists a power series $F_{r,g}(z)$ such that for all complex numbers $z$ with $|z| < 1$
 \begin{equation}\label{ABF2}
 A_{r,g}(z) - (1 - z)^{1/4} B_{r, g}(z) = z^{2r+1 -g}F_{r,g}(z)
 \end{equation}
  and 
  \begin{equation}\label{F2}
 |F_{r,g}(z)| \leq \frac{{r-g+1/4 \choose r+1-g} {r- 1/4 \choose r}}{{2r + 1 - g \choose r}} (1 - |z|)^{-\frac{1}{2}(2r + 1 - g)}.
 \end{equation}
 \item[(ii)] For all complex numbers $z$ with $|1 - z| \leq 1$ we have 
 \begin{equation}\label{A2}
 |A_{r,g}(z)| \leq {2r - g \choose r}.
 \end{equation}
 
 \item[(iii)] For all complex numbers $z \neq 0$ and for $h \in \{1 , 0\}$ we have
 \begin{equation}\label{BA2}
 A_{r, 0}(z) B_{r+h , 1 , 1}(z) \neq A_{r+h , 1}(z) B_{r , 0}(z).
 \end{equation} 
 \end{itemize}
\end{lemma}

\begin{proof} This lemma has been proven in \cite{Akh2}.
\end{proof}

 \section{Equivalent Forms}\label{EF2}
 
 We will call  forms $F_{1}$ and $F_{2}$ equivalent if they are equivalent under $SL_{2}(\mathbb{Z})$-action (i.e. if there exist integers $b$, $c$, $d$ and $e$ such that
$$
F_{1}(bx + c y , dx + ey) = F_{2}(x , y)
$$
for  all $x$ and $y$, where $be - cd = \pm 1$). Denote by $N_{F}$ the number of solutions in integers $x$ and $y$ of the Diophantine equation
\begin{equation*}
\left|F(x , y)\right| = h.
\end{equation*}
 Note that if $F_{1}$ and $F_{2}$ are equivalent, then $N_{F_{1}} = N_{F_{2}}$, $I_{F_{1}} = I_{F_{2}}$ and $J_{F_{1}} = J_{F_{2}}$. 

 Let us define, for a quartic form $F$, an associated  quartic form, the Hessian $H$, by
  $$
  H(x , y) = \frac{d^{2}F}{dx^{2}}  \frac{d^{2}F}{dy^{2}} - \left(\frac{d^{2}F}{dx dy}\right)^{2}. 
  $$
  Then
 $$
 H(x , y)= A_{0}x^{4} + A_{1}x^{3}y + A_{2}x^{2}y^{2} + A_{3}xy^{3} + A_{4}y^{4},
 $$
  where
  \begin{eqnarray}\label{Ai2}\nonumber
 A_{0} & = & 3 (8a_{0}a_{2} - 3a_{1}^{2}) ,\\  \nonumber
 A_{1} & = & 12(6a_{0}a_{3} - a_{1}a_{2}), \\ 
 A_{2} & = &  6(3a_{1}a_{3} + 24a_{0}a_{4} -2a_{2}^{2}), \\  \nonumber
 A_{3} & = & 12(6a_{1}a_{4} - a_{2}a_{3}),\\ \nonumber
 A_{4} & = & 3(8a_{2}a_{4} - 3 a_{3}^{2}).
  \end{eqnarray}

  We have the following identities (see Proposition 5 of \cite{Cre2}):
 \begin{equation}\label{IHF2}
 I_{H} = 12^{2} I_{F}^{2}, 
 \end{equation}
 \begin{equation}\label{JHF2}
     J_{H} = 12^{3}(2I_{F}^{3} - J_{F}^{2})  
     \end{equation}
     and
     $$  D_{H} = 12^{6}J_{F}^{2}D_{F} , 
 $$
  where $H$ is the Hessian of $F$ and $D_{F}$, $D_{H}$ are the discriminants of $F$ and $H$, respectively . From identities in (\ref{Ai2}) and using algebraic manipulation, we have
 \begin{eqnarray*}
& & A_{0}A_{3}^{2} - A_{4}A_{1}^{2} \\ \nonumber
& = & 12^{3}(a_{0}a_{3}^{2} - a_{4}a_{1}^{2}) \left( 2a_{2}^{3} - 9a_{1}a_{2}a_{3} + 27 a_{1}^{2}a_{4} - 72 a_{0}a_{2}a_{4} + 27a_{0}a_{3}^{2}\right)\\ \nonumber
& = & 12^{3}(a_{0}a_{3}^{2} - a_{4}a_{1}^{2})J_{F} 
 \end{eqnarray*}
  and similarly, 
  $$
  A_{3}^{3} + 8A_{1}A_{4}^{2} - 4A_{2}A_{3}A_{4} = 12^{3}(a_{3}^{3} + 8a_{1}a_{4}^{2} - 4 a_{2}a_{3}a_{4})J_{F}.
  $$
  When $J_{F} = 0$, we obtain
  \begin{eqnarray}\label{22}
   A_{0}A_{3}^{2} &=& A_{4}A_{1}^{2},\\ \nonumber
   A_{3}^{3} + 8A_{1}A_{4}^{2} &= &4 A_{2}A_{3}A_{4}.
 \end{eqnarray}
  Therefore,   when $A_{3}A_{4} \neq 0$,
     \begin{eqnarray*}
   H(x , y) & = &A_{0}x^{4} + A_{1}x^{3}y + A_{2}x^{2}y^{2} + A_{3}xy^{3} + A_{4}y^{4} \\
& = &
\frac{1}{4A_{3}^{2}A_{4}}(2A_{1}A_{4}x^{2} + A_{3}^{2}xy + 2
   A_{4}A_{3}y^{2})^{2} \\ \nonumber
   & = & \frac{1}{4A_{3}^{2}A_{4}} W(x , y)^2,  
    \end{eqnarray*}
  where we define the quadratic form $W(x , y) =  2A_{1}A_{4}x^{2} + A_{3}^{2}xy + 2
   A_{4}A_{3}y^{2}$.  So we get
$$
I_{H}= \left(\frac{A_{3}^4 - 16A_{1}A_{4}^2A_{3}}{4A_{3}^2 A_{4}}\right)^2.
$$
From (\ref{IHF2}), we obtain
  \begin{equation}\label{c22}
\left|  A_{3}^4 - 16A_{1}A_{4}^2A_{3}\right| = \left|48A_{3}^2 A_{4}I_{F}\right|.
    \end{equation}

 In order to make  good use of the above identities, we prove the following lemma:
   \begin{lemma}\label{A342}
Let $F(x , y)$ be a quartic form with $J_{F}= 0$. There exists a form equivalent to $F(x , y)$, for which $A_{3} A_{4} \neq 0$.
\end{lemma}
\begin{proof}
 If $A_{4} = 0$, then by (\ref{22}) we have $A_{4} = A_{3} = 0$ and therefore,
  $$
  H(x , y) = x^{2}(A_{0}x^{2} + A_{1}xy + A_{2}y^{2}) . 
  $$
  Let
  $$
  x= mX + lY 
  $$
  and
  $$
  y = pX + q Y, 
  $$
  where $m, l, p$ and $q$ are integers satisfying $mq - lp = \pm 1$. Suppose that $\Phi_{1}(X , Y)$ is equivalent to $F(x , y)$ under this substitution with Hessian
  $$
  H_{\Phi_{1}}(X , Y) =  A'_{0}X^{4} + A'_{1}X^{3}Y + A'_{2}X^{2}Y^{2} + A'_{3}XY^{3} + A'_{4}Y^{4}.
   $$
  We have,
   $$
   A'_{4} = H_{\Phi_{1}} (0 , 1) = H_{F}(l , q)  =l^{2}(A_{0}l^{2} + A_{1}lq + A_{2}q^{2}).
   $$
   If $H_{F}$ is identically zero then by (\ref{IHF2}) and (\ref{JHF2}), we will have $I_{F} = J_{F} = D_{F} =0$. But since
we have  assumed that  $F(x , y)$ is irreducible, $H_{F}(x , y)$ is not identically zero. Therefore, the integers $l$ and $q$ can be chosen so that 
$$
A'_{4} = H_{F}(l , q)  \neq 0.
$$

Let $t \in \mathbb{Z}$ and put 
$$
M = m + lt,
$$
$$
P = p + qt.
$$ 
 Let $\Phi_{2} (X , Y)$ be the equivalent form to $F(x , y)$ under the substitution  $$
 x = MX + lY
 $$
 and
 $$
 y = PX + qY,
 $$
and 
$H_{\Phi_{2}}(X , Y) =   A''_{0}X^{4} + A''_{1}X^{3}Y + A''_{2}X^{2}Y^{2} + A''_{3}XY^{3} + A''_{4}Y^{4}$. Then substituting $x$ by $ MX + lY$ and $y$ by $PX + qY$ in $H_{F}(x , y)$, we find that $A''_{3}$, the coefficient of the term $XY^3$ in $H_{\Phi_{2}}(X , Y)$, is equal to
\begin{eqnarray*}\nonumber
  A''_{3} & =&  4Ml^{3}A_{0} + (l^{3}P + 3Ml^{2}q)A_{1} + (2l^{2}Pq + 2Mlq^{2}) A_{2} \\ \nonumber
& & + (q^{3}m + 3Pq^{2}l)A_{3} + 4Pq^{3} A_{4} \\ \nonumber
& =& (m + lt)(4l^{3}A_{0} + 3l^{2}qA_{1} + 2lq^{2}A_{2} + q^{3} A_{3}) \\ \nonumber
& & + (p + qt) (l^{3}A_{1} + 2l^{2}qA_{2} + 3lq^{2}A_{3} +4q^{3}A_{4}) \\ \nonumber
& =& K + 4t (l^4 A_{0} + l^3 q A_{1} + l^2 q^2 A_{2}+ l q^3 A_{3}+ q^4 A_{4})\\ \nonumber
& = & K + 4t A'_{4}.
  \end{eqnarray*}
Since $A'_{4} \neq 0$, the integer $t$ can be chosen so that $A''_{3} \neq 0$.
\end{proof}
  In the following, we will show that $F(x , y)$ or one of its equivalences (under $GL_{2}(\mathbb{Z})$-action) satisfies 
  $$
  |A_{4}| < 4I .
  $$
  From now on, we will suppose that $A_{3}A_{4} \neq 0$. Let
 $$
 x = mX + lY
 $$
 and
 $$
 y = pX + qY ,
 $$
 where $m$, $l$, $p$ and $q$ are integers satisfying $mq - lp = \pm 1$. Let $\Phi( X , Y)$ be equivalent to $F(x , y)$ under this substitution and
  $$
  \Phi(X , Y) = a'_{0}X^{4} + a'_{1}X^{3}Y + a'_{2}X^{2}Y^{2} + a'_{3}X Y^{3} + a'_{4} Y^{4} .
  $$
 We observe that
 $$
 A'_{4} = H_{\Phi} (0 , 1) = H_{F}(l , q) ,
 $$
  where $H_{\Phi}( X , Y) =   A'_{0}X^{4} + A'_{1}X^{3}Y + A'_{2}X^{2}Y^{2} + A'_{3}XY^{3} + A'_{4}Y^{4} $.
 
 To continue,  we will be in need of  the following Proposition due to Hermite.
  
  \begin{prop}\label{her2}
   Suppose that $f_{11}x^{2} + 2f_{1 2} xy + f_{2 2}y^{2}$ is a binary form with $D = f_{11}f_{22} - f_{1 2}^{2} \neq 0$. Then there is an integer pair $(u_{1} , u_{2}) \neq (0 , 0)$ for which
  $$
  0 < |f_{1 1}u_{1}^{2} + 2f_{1 2} u_{1}u_{2} + f_{2 2}u^{2} | < \sqrt{\frac{4}{3}|D|} .
  $$
  \end{prop}
  \begin{proof}
  See \cite{Cas2}, page 31.
 \end{proof}
 
  Proposition \ref{her2} implies that  we can choose $l$ and $q$,  such that
  \begin{eqnarray}\label{Rr2}\nonumber
 0 <|A'_{4}| &=& \frac{1}{|4A_{3}^{2}A_{4}|} (2A_{1}A_{4}l^{2} + A_{3}^{2}lq + 2A_{4}A_{3}q^{2})^{2} \\  \nonumber
 & < &   \frac{1}{|4A_{3}^{2}A_{4}|} \left|\frac{1}{3}(A_{3}^{4} - 16A_{1}A_{4}^{2}A_{3}) \right| \\ 
& = & 4|I|,
 \end{eqnarray}
where the last equality comes from (\ref{c22}).

We have shown that the Hessian of $F$ satisfies the following formula.
\begin{eqnarray*}
H(x , y) & =& A_{0}x^{4} + A_{1}x^{3}y + A_{2}x^{2}y^{2} + A_{3}xy^{3} + A_{4}y^{4} \\ \nonumber
& = & \frac{1}{4A_{3}^{2}A_{4}}(2A_{1}A_{4}x^{2} + A_{3}^{2}xy + 2A_{4}A_{3}y^{2})^{2}.
\end{eqnarray*} 
We will need some results due to Cremona \cite{Cre2}. Since we are using different notations in this paper, we will summarize Propositions $6$ and $8$  of \cite{Cre2} in Lemmas \ref{P682} and \ref{P882}. In particular, we note that the quartic polynomial $g_{4}(X)$ in \cite{Cre2} is equal to $\frac{-1}{3}H(x , 1)$ and its leading coefficient  is equal to $-A_{0}/3$.
\begin{lemma}\label{P682}
 Suppose $F(x , y)$ is a quartic form with invariants $I$ and $J$ and Hessian $H(x , y)$. Let $\phi$ be a root of $X^3 -3I + J$. Then 
$$
-\frac{1}{9}H(x , y) + \frac{4}{3} \phi F(x , y) = m(x , y)^2,
$$
where $m(x , y)$ is a quadratic covariant of $F(x , y)$.
\end{lemma}
\begin{proof}
See part (vi) of Proposition $6$ of \cite{Cre2}.
\end{proof}

\begin{lemma}\label{P882}
Let F(x , y) be a quartic form with real coefficients and the leading coefficient $a_{0}$. Suppose that $F(x , 1) = 0$ has $4$ real roots. Order the roots $\phi_{i}$ of $X^3 -3I + J$, with $4a_{0} \phi_{1} > 4a_{0} \phi_{2} > 4a_{0} \phi_{3}$. Set $\phi = \phi_{2}$. Then $m(x , y)$ is a positive definite quadratic form with real coefficients, where $m(x , y)$ is the covariant of $F(x , y)$ defined in Lemma $\ref{P682}$.
\end{lemma} 
\begin{proof}
This is part (ii) of Proposition $8$ of \cite{Cre2}. Note that the quantity $z$ in that Proposition  is equal to $-A_{0}$ and therefore a positive value in our case.
\end{proof}

Following Definition $4$ of \cite{Cre2},  we say that the quartic form $F(x , y) = a_{0}x^{4} + a_{1}x^{3}y + a_{2}x^{2}y^{2} + a_{3}x y^{3}+ a_{4}y^{4}$ with positive discriminant is \emph{reduced} if and only if the positive definite quadratic form  $m(x ,y)$ is reduced. Here, we remark that the real quadratic form $f(x , y) = ax^{2} + bxy + cy^{2}$ is called \emph{reduced} if 
 $$
 |b| \leq a \leq c .
 $$

\begin{lemma}\label{red2}
Let $F$ be the quartic form  in Theorem $\ref{main22}$ and $H$ be its Hessian. If $F$ is reduced then for integers $x, y$ we have
 $$
 \left|H(x , y)\right| \geq 36 \, I y^4.
 $$
 \end{lemma}
\begin{proof}
Suppose that our quartic form $F(x , y)$  is reduced.  Taking $\phi = 0$ in Lemma \ref{P682}, we know that the algebraic covariant $\frac{-1}{9} H(x , y)$
is the square of a quadratic form, say
$$
\frac{-1}{9} H(x , y) = m^2(x , y).
$$
  We assume that $y \neq 0$. Put 
 $$
 m(x , y) = y^2 m(z) = y^2 \left(Az^2 + B z + C \right),
 $$
 where $z = \frac{x}{y}$.
 Note that  $m(z)$ assumes a minimum equal to $ \frac{4AC - B^2}{4A} $ at $ z = \frac{-B}{2A}$. Since 
$$m^{2}(x , y) = \frac{1}{36A_{3}^2A_{4}}(2A_{1}A_{4}x^{2} + A_{3}^{2}xy + 2A_{4}A_{3}y^{2})^{2},$$
 by (\ref{c22}), we get 
 $$
 4AC - B^2 = \frac{16A_{1}A_{3}A_{4}^2 - A_{3}^4}{-36A_{3}^2 A_{4}} = \pm\frac{4}{3}I .
 $$
Recall that  $A_{0} < 0 $ and hence, by (\ref{22}), $A_{4} < 0$. 
 Since $I > 0$ and $m(x , y)$ is reduced, we have $4AC - B^2 > 0$ and
  $$
 A^2 \leq AC \leq \frac{1}{3} (4AC - B^2) = \frac{4}{9} I.
 $$
 Therefore, 
 $$
 m(x , y) \geq 2\sqrt{I}y^2,
 $$
\end{proof}
 So we can assume that $\left| H(x , y)\right| \geq h^{3} 12\sqrt{3I}$ when looking for pairs of solutions $(x , y)$ with $ |y| \geq \frac{h^{3/4}}{ (3I)^{1/8}}$.

\section{Reduction To A Diagonal Form}\label{RDF2}
Our goal in this section will be  to  reduce the problem at hand to consideration of diagonal forms over a suitable imaginary quadratic field. The method of Thue-Siegel is particularly well suited for application to such forms. We will show that
\begin{lemma}\label{rF2}
Let $F$ be the binary form in Theorem $\ref{main2}$.  Then 
$$
F(x , y) = \frac{1}{96 A_{3}^2 A_{4}\sqrt{-3I}} \left(\xi^{4}(x , y) - \eta^{4}(x , y) \right),
$$
where $\xi$ and $\eta$ are complex conjugate linear forms in $x$ and $y$.
\end{lemma}
\begin{proof}
 Let  $H(x , y) = A_{0}x^{4} + A_{1}x^{3}y + A_{2} x^{2}y^{2} + A_{3} xy^{3} + A_{4}y^{4} $ with $A_{3}A_{4} \neq 0,$ be the Hessian of $F(x , y)$.
We can factor $2A_{1}A_{4}x^{2} + A_{3}^{2}xy + 2A_{4}A_{3}y^{2}$ over $\mathbb{C}$ as
\begin{equation}\label{HaHa}
 \xi (x , y)\eta(x , y) = 2A_{1}A_{4}x^{2} + A_{3}^{2}xy + 2A_{4}A_{3}y^{2},
 \end{equation}
where $\xi$ and $\eta$ are linear forms. So we may write
$$
  x = m \xi + l \eta 
  $$
  $$y = p\xi + q \eta,$$
for some $m$, $l$, $p$, $q  \in \mathbb{C}$. Therefore,
 \begin{eqnarray*}
 F(x , y) & = &F(m\xi +l\eta , p\xi + q\eta) \\
&= &a'_{0}\xi^{4} + a'_{1}\xi^{3}\eta + a'_{2}\xi^{2}\eta^{2} +a'_{3}\xi \eta^{3} + a'_{4}\eta^{4} \\
& = & \Phi(\xi , \eta).
  \end{eqnarray*}
 The Hessian $H'(\xi , \eta)$ of  $\Phi(\xi , \eta)$ satisfies
 \begin{eqnarray*}
 H'(\xi , \eta)  & = & A'_{0}\xi^{4} + A'_{1}\xi^{3}\eta + A'_{2}\xi^{2}\eta^{2} +A'_{3}\xi \eta^{3} + A'_{4}\eta^{4} \\ \nonumber
  &= &\Delta^{2} H(x , y) = \frac{\Delta^{2}}{4A_{3}^{2}A_{4}}\xi^{2} \eta^{2} .
 \end{eqnarray*}
 Hence,
 \begin{equation}\label{H'2}
 A'_{0} = A'_{1} = A'_{3} = A'_{4} = 0 ; \    \    A'_{2} = \Delta^{2} \frac{1}{4A_{3}^{2}A_{4}} .
 \end{equation}
 
 On the other hand, 
   $$
   A'_{0} = 3(8a'_{0}a'_{2} - 3 a^{'2}_{1}) 
   $$
   and
 $$
 A'_{1} = 12(6a'_{0}a'_{3} - a'_{1}a'_{2}).
 $$
  Using Maple, it  is easy to check that for any form $F(x , y)$, 
   $$
   -10a_{4}A_{0} + 2a_{3} A_{1} - a_{2}A_{2} + a_{1}A_{3} -2a_{0}A_{4} = 6J .
   $$
  So, for $\Phi(\xi , \eta)$, we obtain
 $$
 -10a'_{4}A'_{0} + 2a'_{3} A'_{1} - a'_{2}A'_{2} + a'_{1}A'_{3} -2a'_{0}A'_{4} = 6J_{\Phi} = 6\Delta^{4}J_{F} = 0 ,
 $$
where $a'_{i}$ are the coefficients of $\Phi$ and $A'_{i}$ are the coefficients of its Hessian.
Therefore, by (\ref{H'2}),
$$
a'_{2} = 0 
$$ 
and from the expressions for $A'_{0}$  and $A'_{4}$ respectively that result from  (\ref{Ai2}),
$$
a'_{1} = a'_{3} = 0 ,
$$
whereby,
$$
F(x , y) = \Phi (\xi , \eta) = a'_{0}\xi^{4} + a'_{4}\eta^{4}. 
$$
Observe that if 
$$ 2A_{1}A_{4}x^{2} + A_{3}^{2}xy + 2A_{4}A_{3}y^{2} =  (\alpha x + \beta y) (\gamma x + \delta y)$$
then for any complex number $\lambda$, in (\ref{HaHa}) we may take 
$\xi = \lambda (\alpha x + \beta y)$ and   $\eta =\mu (\gamma x + \delta y)$, where  $\lambda \mu =1$.
Our goal now is to determine the values of $\lambda$ and $\mu = \frac{1}{\lambda}$ in  $\xi =\lambda (\alpha x + \beta y)  $ and $\eta = \mu (\gamma x + \delta y )$, so that $a'_{4} = - a'_{0}$.  We have 
$$
  \left( \begin{array}{cc}
  \lambda \alpha & \lambda \beta \\
   \mu \gamma  & \mu \delta
 \end{array} \right) 
  \left( \begin{array}{cc}
 m  & l \\
  p & q
 \end{array} \right) = \left( \begin{array}{cc}
1 & 0\\
0 & 1
\end{array}\right).
 $$
Thus,
\begin{equation}\label{matrr2}
\left( \begin{array}{cc}
 m  & l \\
  p & q
 \end{array} \right) = \left( \begin{array}{cc}
  \lambda \alpha & \lambda \beta \\
   \mu \gamma  & \mu \delta
 \end{array} \right)^{-1}
=
\frac{1}{\lambda \mu (\alpha \delta - \beta \gamma)} \left( \begin{array}{cc}
 \mu \delta  & -\lambda \beta \\
   -\mu \gamma  & \lambda \alpha
 \end{array} \right)  
\end{equation} 
 whereby,
$$
\frac{p}{q} = \frac{-\mu \gamma}{\lambda  \alpha}
$$
and we get
 $$
 q = -\frac{\lambda p \alpha}{\mu \gamma}.
 $$
Since $\Phi (\xi , \eta) = a'_{0}\xi^{4} + a'_{4}\eta^{4}$, we have
$$
a'_{0} = \Phi(1 , 0) \qquad \textrm{and} \qquad a'_{4} = \Phi(0 , 1).
$$
When $\eta = \mu (\gamma x + \delta y ) = 0$, we have 
$$m = \frac{-\delta p}{\gamma}
$$
and when $\xi =\lambda (\alpha x + \beta y) = 0 $, we have
$$
l = \frac{-\beta q}{\alpha}.
$$
So we can write
$$
a'_{0} = F(m , p)= F\left(\frac{-\delta p}{\gamma} , p\right) = \frac{p^{4}}{\gamma^{4}}F(-\delta , \gamma)
$$  
and
$$
a'_{4} =F (l , q)= F\left(\frac{-\beta q}{\alpha} , q\right) = \frac{q^{4}}{\alpha^{4}}F(-\beta , \alpha).
$$
Therefore, if we choose  $\lambda$ and $\mu$ so that $\mu ^{8} = \frac{\mu^{4}}{\lambda^{4}} = \frac{F(-\beta , \alpha)}{F(\delta , -\gamma)}$, then
$-a'_{0} = a'_{4}$.

We have shown that $F(x , y)$ can be written as $a'_{0}\left(\xi^{4}(x , y) - \eta^{4}(x , y)\right)$, where 
\begin{equation}\label{lin2}
 \xi = \lambda (\alpha x + \beta y), \, \,   \eta =\mu (\gamma x + \delta y)
 \end{equation}
  and $\lambda \mu =1$. It remains to calculate the value of $a'_{0}$.
Using (\ref{H'2}) and (\ref{Ai2}), we get
$$A'_{2} = \Delta^{2} \frac{1}{4A_{3}^{2} A_{4}}= 6(3a'_{1}a'_{3} + 24 a'_{0}a'_{4} -2a'_{2}) = 144a'_{0}a'_{4}.
$$
Substituting  $a'_{4}$ by $ -a'_{0}$, we obtain
$$
a^{'2}_{0} = -\frac{\Delta^{2}}{24^{2}A_{3}^{2}A_{4}},
$$
where $\Delta = mq - lp$ is the determinant of the matrix 
$\left( \begin{array}{cc}
 m  & l \\
  p & q
 \end{array} \right)$. Therefore, from (\ref{matrr2}) and he fact that $\lambda \mu = 1$,
$$
 a^{'2}_{0} = -\frac{1}{(\alpha \delta -  \beta\gamma)^{2} 24^{2}A_{3}^{2}A_{4}}.
$$
To calculate $(\alpha \delta -  \beta \gamma)^{2}$, we recall that 
$$
2A_{1}A_{4}x^{2} + A_{3}^{2}xy + 2A_{3}A_{4}y^{2} = (\alpha x + \beta y) (\gamma x + \delta y) ,
$$
consequently, computing the discriminant of the above quadratic form and by (\ref{c22}),
\begin{equation}\label{IA2}
\left|\alpha \delta -  \beta \gamma\right|^{2} = \left| A_{3}^{4} - 16 A_{1}A_{4}^{2}A_{3}\right| =  \left|48A_{3}^{2}A_{4}I \right|
\end{equation}
and therefore,
\begin{equation*}
 a'_{0} = \pm \frac{1}{96 A_{3}^{2}A_{4}\sqrt{-3I}} ,
 \end{equation*}
where $I = I_{F}$.

We will assume, without loss of generality, that
\begin{equation}\label{a'2}
 a'_{0} =  \frac{1}{96 A_{3}^{2}A_{4}\sqrt{-3I}}.
 \end{equation}
\end{proof}

\section{Resolvent Forms}\label{resolvent2}

Suppose that $\xi$ and $\eta$ are linear forms in Lemma \ref{rF2}. Let us define
  $$\xi' = \frac{\xi}{(12A_{3}^{2})^{1/4}|A_{4}|^{1/8} } $$ 
 and   
  $$\eta' = \frac{\eta}{(12A_{3}^{2})^{1/4} |A_{4}|^{1/8}} ,$$
so that
$$
F(x , y) = \frac{1}{8\sqrt{3IA_{4}}} \left(\xi^{'4}(x , y) - \eta^{'4}(x , y) \right).
$$
 Lemma \ref{rF2} can   be restated as follows: 
 \begin{lemma}\label{nr2}
Let $F$ be the binary form in Theorem $\ref{main2}$.  Then 
\begin{equation}\label{aa2}
F(x , y) = \frac{1}{8\sqrt{3IA_{4}}} \left(\xi^{4}(x , y) - \eta^{4}(x , y) \right).
\end{equation}
where $\xi$ and $\eta$ are complex conjugate linear forms in $x$ and $y$, with
     $$
    \xi^4, \eta^4  \in \mathbb{Q} \left(\sqrt{A_{0}I/3} \right).   
   $$
\end{lemma}
\begin{proof}
For the binary form $F(x , y)$ with Hessian $H(x , y)$, the sextic covariant  $Q(x , y)$ is defined by
$$
Q(x , y) = \frac{\delta F}{\delta x} . \frac{\delta H}{\delta y} - \frac{\delta F}{\delta y}. \frac{\delta H}{\delta x}.
$$
Since we have taken $H(x , y) =\frac{1}{4A_{3}^{2}A_{4}}(2A_{1}A_{4}x^{2} + A_{3}^{2}xy + 2A_{4}A_{3}y^{2})^{2} $, we may write
$$
Q(x , y) = \frac{1}{2A_{3}^{2}A_{4}} W(x , y) \psi(x ,y) ,
$$
where 
$$
 W(x , y) =  2A_{1}A_{4}x^{2} + A_{3}^{2}xy + 2A_{4}A_{3}y^{2}
 $$
and
  $$
  \psi(x , y) = (A_{3}^{2}x + 4A_{3} A_{4}y) \frac{\delta F} {\delta x} - (4A_{1}A_{4}x + A_{3}^{2}y) \frac{\delta F}{\delta y}. 
  $$
We have (see equation (25) of \cite{Cre2})
 $$
 16 H^{3} + 9 Q^{2} = 4^{4} \times 3^{3}   I H F^{2} .
 $$
 We remark that in \cite{Cre2}, $g_{4} = \frac{-1}{3}H$, $g_{6} =\frac{-1}{36} Q$ and the invariants $I$ and $J$ are the negative of our $I$ and $J$, respectively.
 Since $H(x , y) = \frac{1}{4A_{3}^{2}A_{4}}W(x , y)^2$ is not identically zero, we can divide both sides of the obove identity by $H(x , y)$ to get
  \begin{equation}\label{42}
 W^{4}(x , y) + 9 A_{3}^{2}A_{4} \psi ^{2}(x , y) = 4^{4} \times 3^{3}  I A_{3}^{4}A_{4}^{2} F^{2}(x , y) .
  \end{equation}
  Since $W(x , y) = \xi  \eta $ and  $F(x , y) = \frac{ \xi^{4} - \eta^{4}}{96 A_{3}^2 A_{4}\sqrt{-3I}} $,  (\ref{42}) implies that
$$
\xi^4  \eta^4 + 9 A_{3}^{2}A_{4} \psi ^{2}(x , y) = \frac{1}{4} \left(-\xi^{8} - \eta^{8} + 2 \xi^4  \eta^4 \right)
$$
and we obtain
  \begin{equation}\label{Q2}
 (\xi^{4} + \eta^{4})^{2} = -36 A_{3}^{2}A_{4} \psi^{2}(x , y)
 \end{equation}
 and therefore, by (\ref{22})
 $$
 \xi^{4} + \eta^{4}=\pm \frac{ 6 A_{3}^{2}}{A_{1}}\sqrt{-A_{0}}\psi(x , y) .
 $$  
  Note that if all roots of $F(x , 1)$ are real then $I > 0$ and $A_{0} <0$ ( see \cite{Cre2}, Proposition 7). So we may write  
  $$
 \xi^{4} + \eta^{4} = b\sqrt{-A_{0}},
 $$
 with $b \in \mathbb{Q}$.  
 We have also seen that
 $$
 \xi^{4} - \eta^{4} = ia\sqrt{3I} 
 $$
  for some even integer $a$. Therefore, for integers $x$, $y$, the quantities $\xi^{4}(x , y)$ and $\eta^{4}(x ,y)$ are complex conjugates and belong to  $\mathbb{Q}\left(\sqrt{-A_{0}} , \sqrt{-3I}\right)$. Moreover, $\sqrt{-A_{0}}\,\xi^{4}(x , y)$ and $\sqrt{-A_{0}}\,\eta^{4}(x ,y)$ are   algebraic integers in $\mathbb{Q}\left(\sqrt{A_{0}I/3} \right)$. This is because  
$$
 \sqrt{-A_{0}} \left(\xi^{4} + \eta^{4}\right) =\pm \frac{ -6 A_{0} A_{3}^{2}}{A_{1}}\psi(x , y) .
 $$
and by (\ref{22}),  $A_{1} | A_{0}A_{3}^2$.  We will work in the number field $\mathbb{Q}\left(\sqrt{A_{0}I/3} \right)$.
    We also have
 \begin{eqnarray*}
 \frac{\xi^{4}}{\eta^{4}} & = & \frac{b\sqrt{-A_{0}} + ia\sqrt{3I}} {b\sqrt{-A_{0}} - ia\sqrt{3I}} \\ \nonumber
 & = & \frac{-A_{0}b^{2} -3a^{2} I + i6ab\sqrt{-A_{0}I/3}}{-A_{0}b^{2} + 3a^{2}I} 
 \end{eqnarray*}
 Therefore,
 $$
  \frac{\xi^{4}}{\eta^{4}} \in \mathbb{Q}(\sqrt{A_{0}I/3}) .
 $$
 Note that, in (\ref{lin2}), we started with two linear forms and continued with their fourth powers. Let  the linear form  $\xi = \xi(x , y)$ be a fourth root of $\xi^{4}(x , y)$ and define    
 $$
 \eta(x , y)= \bar{\xi}(x , y).
 $$
 Indeed, $\eta(x , y)$ is a fourth root of $\eta^4$.
  Hence, when $F(x , 1)$ splits in $\mathbb{R}$, we can define the complex conjugate linear forms $\xi(x , y)$ and $\eta(x , y)$, so that 
 $$
 \xi^{4}  - \eta^{4} = 96A_{3}^{2} A_{4}\sqrt{-3I}  F(x , y)  
 $$
 and 
 $$
\left| \xi \eta\right| = \left| 2A_{1}A_{4}x^{2} + A_{3}^{2}xy + 2A_{4}A_{3}y^{2}\right| .
 $$            
Now let us define
  $$\xi' = \frac{\xi}{(12A_{3}^{2})^{1/4}|A_{4}|^{1/8} } $$ 
 and   
  $$\eta' = \frac{\eta}{(12A_{3}^{2})^{1/4} |A_{4}|^{1/8}} ,$$
so that
$$
F(x , y) = \frac{1}{8\sqrt{3IA_{4}}} \left(\xi^{'4}(x , y) - \eta^{'4}(x , y) \right).
$$
 From (\ref{Ai2}), for every pair of integers $(x , y)$, we have
$$
3 \mid  \frac{1}{4A_{3}^{2}A_{4}} W^{2}(x , y) = H(x , y).
$$
 This gives
  $$
  12 A_{3}^{2} A_{4} \mid  W^{2}(x , y).
  $$
   By (\ref{42}),  for every pair of integers $(x , y)$, we have
  $$
 16 A_{3}^{2} A_{4} \mid \psi^{2} (x , y) . 
 $$
  Using (\ref{Q2}), we conclude  that  the real part of $\xi^{4}$ has the factor
 $12A_{3}^{2} A_{4}$ .  Since $\xi^4 - \eta^4 = a'_{0}F$, by (\ref{a'2}), the imaginary part of $\xi^{4}$ has also the factor $12A_{3}^{2} A_{4}$.
  So 
  $$
   \frac{\xi^4}{|12A_{3}^{2}A_{4}|} ,    \frac{\eta^4}{|12A_{3}^{2}A_{4}|}  \in \mathbb{Q} \left(\sqrt{-A_{0}} , \sqrt{-3I}\right).
   $$   
   By (\ref{22}), 
     $$
   \frac{\sqrt{-A_{4}}\, \xi^4}{|12A_{3}^{2}A_{4}|} ,    \frac{\sqrt{-A_{4}}\eta^4}{|12A_{3}^{2}A_{4}|}  \in \mathbb{Q} \left(\sqrt{A_{0}I/3} \right).   
   $$
\end{proof}
 We call a pair of complex conjugates $\xi$ and $\eta$ satisfying the  identities in Lemma \ref{nr2} a pair of \emph{resolvent forms}, and note that if $(\xi , \eta)$ is one pair, there are precisely three others, given by $(i \xi , -i \eta)$ , 
$(-\xi , -\eta)$ and $(-i\xi , i\eta)$, where $i = \sqrt{-1}$. 
We will, however, work with $(\xi , \eta)$, a fixed pair of resolvent forms. For the pair of resolvent form   $(\xi , \eta)$, we have
\begin{equation}\label{c62}
|\xi \eta| =  \left|\frac{2A_{1}A_{4}x^{2} + A_{3}^{2}xy + 2A_{4}A_{3}y^{2}}{\sqrt{12 A_{3}^2 \sqrt{\left|A_{4}\right|}}}\right| = \frac{\left(H(x , y)^2|A_{4}|\right)^{1/4}}{\sqrt{3}} .
\end{equation}
\bigskip

{\bf REMARK}. The fact that for integers $x$ and $y$, $\xi^{4}(x , y)$ and $\eta^4(x , y)$ are complex conjugates and belong to an imaginary quadratic field is very crucial for our proof. To satisfy these conditions, when $J_{F} =0$, we only need $I_{F}A_{0} <0$ (see the proof of Lemma \ref{nr2}). Proposition $7$ of 
\cite{Cre2}, guarantees this property for  quartic binary forms that split in $\mathbb{R}$. So we may generalize Theorem \ref{main2} to all quartic binary forms with $I_{F}A_{0} <0$.

 \section{Gap Principles} \label{GPe2}

Let $\omega$ be a fourth root of unity (for some $j \in \{ 1 , 2,  3, 4\}$, let  $\omega= e^{\frac{2j\pi i}{4}}$). We say that the integer pair $(x , y)$ is \emph{related} to $\omega$ if 
$$
\left|\omega - \frac{\eta(x , y)}{\xi(x , y)}\right| = \min_{0 \leq k \leq 3}\left|e^{2k\pi i/4} - \frac{\eta(x , y)}{\xi(x , y)}\right|. 
$$
 Let us define  $z = 1 - \left(\frac{\eta(x , y)}{\xi(x , y)}\right)^{4}$, where $(\xi , \eta)$ is a fixed pair of resolvent forms ( in other words, $\frac{\eta}{\xi}$ is a fourth root of  $(1-z)$).
We have
$$
 |1 - z| = 1  \    \   ,  \    \     |z| < 2 .
 $$
Note that $|z| = 2$ is impossible here. Because it would mean $\eta^4 = -\xi^4$, so $F(x , y) = \frac{1}{4 \sqrt{3IA_{4}}} \xi^4$ and hence $D_{F} = 0$.
 \begin{lemma}\label{6.12}
 Let  $\omega$ be a fourth root of unity and the integral pair $(x , y)$ satisfies $F(x , y) =  \frac{1}{8\sqrt{3IA_{4}}}(\xi^{4}(x , y) - \eta^{4}(x , y)) =1$, with
 $$
\left|\omega - \frac{\eta(x , y)}{\xi(x , y)}\right| = \min_{0 \leq k \leq 3}\left|e^{2k\pi i/4} - \frac{\eta(x , y)}{\xi(x , y)}\right|. 
$$
 If $|z| \geq 1$ then 
 \begin{equation}\label{Gap12}
\left|\omega  - \frac{\eta(x,y)}{\xi(x,y)}\right| \leq \frac{\pi}{8} |z|.
\end{equation}
If $|z| < 1$ then 
\begin{equation}\label{Gap22}
\left|\omega - \frac{\eta(x,y)}{\xi(x,y)}\right| < \frac{\pi}{12} |z|.
\end{equation}
  \end{lemma}
\begin{proof}
 Put
 $$
 4\theta = \textrm{arg} \left( \frac{\eta(x , y)^{4}}{\xi(x , y)^{4}}\right).
 $$ 
  We have
  $$
  \sqrt{2 - 2\cos(4\theta)} = |z| .
  $$
  Therefore, when $|z| < 2$ we have 
$$|\theta| < \frac{\pi}{4} $$
and  when $|z| < 1$ we have
$$|\theta| < \frac{\pi}{12}.$$
 Since 
 $$ 
\left|\omega - \frac{\eta(x,y)}{\xi(x,y)} \right| \leq  |\theta|, 
 $$
we obtain
$$
\left|\omega - \frac{\eta(x,y)}{\xi(x,y)}\right| \leq \frac{1}{4} \frac{|4\theta|}{\sqrt{2 - 2\cos(4\theta)}} 
\left|1 - \frac{\eta(x,y)^{4}}{\xi(x,y)^{4}}\right|. 
$$
By differential calculus $\frac{|4\theta|}{\sqrt{2 - 2\cos(4\theta)}} < \frac{\pi}{2}$ whenever $ 0<|\theta| < \frac{\pi}{4} $. Therefore
 $$
 \left|\omega - \frac{\eta(x , y)}{\xi(x , y)}\right| <\frac{\pi}{8} |z| ,
 $$
 and from the fact that $\frac{|4\theta|}{\sqrt{2 - 2\cos(4\theta)}} < \frac{\pi}{3}$ whenever $ 0<|\theta| < \frac{\pi}{12} $ , we conclude
 $$
 \left|\omega - \frac{\eta(x , y)}{\xi(x , y)}\right| <\frac{\pi}{12} |z| ,
 $$
as desired. 
\end{proof}

Suppose that we have distinct solutions to $\left| F(x , y)\right| \leq  h$ indexed by $i$, say $(x_{i}, y_{i})$, related to a fixed fourth root of unity $\omega$ with $|\xi(x_{i+1} , y_{i+1})| \geq |\xi(x_{i} , y_{i})|$. Let
$$
F(x_{i} , y_{i}) = h_{i}, \, \,  F(x_{i+1} , y_{i+1}) = h_{i+1}.
$$
For brevity, we will write $\eta_{i} = \eta(x_{i} , y_{i})$ and $\xi_{i} = \xi(x_{i} , y_{i})$. 
We have
\begin{displaymath}
\left(\begin{array}{cc}
\lambda \alpha & \lambda \beta \\
\mu \gamma & \mu \delta 
\end{array}\right)
\left(\begin{array}{cc}
x_{1} & x_{2} \\
y_{1} & y_{2}
\end{array} \right)
=
\sqrt{12 A_{3}^2 \sqrt{\left|A_{4}\right|}} \left(\begin{array}{cc}
\xi_{1} & \xi_{2} \\
\eta_{1} & \eta_{2}
\end{array}\right)
\end{displaymath}
(see the definition of linear forms $\xi$ and $\eta$ in Section \ref{resolvent2}).
 Since $(x_{1}, y_{1})$ and $(x_{2}, y_{2})$ are distinct co-prime solutions,    $x_{1}y _{2}- x_{2}y_{1}$ is a nonzero integer. So by  (\ref{IA2}) and (\ref{22}), we get 
 \begin{equation}\label{lb2}
 |\xi_{1} \eta_{2} - \xi_{2} \eta_{1} | = \frac{|(\alpha \delta - \beta \gamma)(x_{1}y _{2}- x_{2}y_{1}) |}{\sqrt{12 A_{3}^2 \sqrt{\left|A_{4}\right|}} }  \geq 2\sqrt{I} \left| A_{4} \right|^{1/4}.
 \end{equation}    
On the other hand, by (\ref{Gap12}) and  (\ref{Gap22}),  we have
 \begin{eqnarray*}
|\xi_{i} \eta_{i+1} - \xi_{i+1} \eta_{i}| & = &\left|\xi_{i}( \eta_{i+1} - \omega \xi_{i+1}) - \xi_{i+1}( \eta_{i} - \omega \xi_{i})\right| \\ \nonumber 
& \leq & \left| \xi_{i}\xi_{i+1} \left(\frac{\eta_{i+1}}{\xi_{i+1}} -\omega \right)\right| + \left| \xi_{i}\xi_{i+1} \left(\frac{\eta_{i}}{\xi_{i}} -\omega \right)\right|\qquad (\textrm{by the triangle inequality})\\ \nonumber
&\leq & \frac{\pi}{8}\left(| \xi_{i}\xi_{i+1}z_{i+1}| + | \xi_{i}\xi_{i+1}z_{i}|\right) \qquad  (\textrm{from} (\ref{Gap12}))\\ \nonumber
& = & \frac{\pi}{8}\left(| \xi_{i}\xi_{i+1} \frac{\eta_{i+1}^4 - \xi_{i+1}^4}{\xi_{i+1}^4}| + | \xi_{i}\xi_{i+1}\frac{\eta_{i}^4 - \xi_{i}^4}{\xi_{i}^4}|\right) \\ \nonumber
&\leq& \pi h\sqrt{\left|3I\, A_{4}\right|}\left(\frac{|\xi_{i}|}{|\xi_{i+1}^{3}|} + \frac{|\xi_{i+1}|}{|\xi_{i}^{3}|}\right),
\end{eqnarray*}
the last inequality holding from expression for $F(x , y)$ in Lemma \ref{nr2} and since $\left|F(x , y) \right| < h$.
Since we assumed $|\xi_{i}| \leq |\xi_{i+1}|$, we get
$$
|\xi_{i} \eta_{i+1} - \xi_{i+1} \eta_{i}| \leq 2 \pi h\sqrt{\left|3I\, A_{4}\right|} \left( \frac{|\xi_{i+1}|}{|\xi_{i}^{3}|}\right)
$$
Combining this with (\ref{lb2}), we conclude 
\begin{equation}\label{Gap2}
|\xi_{i+1}| \geq \frac{1}{\pi \sqrt{3}h\, \left|A_{4}\right|^{1/4}}|\xi_{i}|^{3}.
\end{equation}

Let us now assume that there are  $4$ distinct solutions to $\left|F(x , y) \right| \leq h$ related to a fixed choice of $\omega$, corresponding to $\xi_{-1}$, $\xi_{0}$, $\xi_{1}$ and $\xi_{2}$, where $|\xi_{-1}| \leq |\xi_{0}| \leq  |\xi_{1}| \leq |\xi_{2}| $and $F(x_{i} , y_{i}) = h_{i} $. We will deduce a contradiction, which shows that at most $3$ such solutions can exist.  
By (\ref{Gap2}) and since $\left|h_{i}\right| \leq h$,
$$
|z_{i+1}| \leq \frac{ 3 \pi^4 |z_{i}|^{3} h^2}{64 I} , 
$$
where $z_{i} = 1 - \frac{\eta_{i}^{4}}{\xi_{i}^{4}} = \frac{8h \sqrt{ \left|3I\, A_{4}\right|}}{\xi_{i}^{4}}$.  Since $|z_{-1}|\leq 2$, if $I > 36.6 h^2$ then $|z_{0}|$ , $|z_{1}|$ , $ |z_{2}| < 1$.   By (\ref{Gap12}),
\begin{eqnarray*}
|\xi_{-1} \eta_{0} - \xi_{0} \eta_{-1}|  &=& |\xi_{-1}(\omega \eta_{0} - \xi_{0}) - \xi_{0}( \omega \eta_{-1} - \xi_{-1})| \\
& \leq & 8 h (1 + \frac{\pi}{12})\sqrt{ \left|3I\, A_{4}\right|} \left( \frac{|\xi_{0}|}{|\xi_{-1}^{3}|}\right). 
\end{eqnarray*}
 Combining this with (\ref{lb2}), we conclude 
 $$
  |\xi_{0}| \geq \frac{2\sqrt{3}}{5 \pi h\, \left|A_{4}\right|^{1/4}}|\xi_{-1}|^{3} .$$
Similarly, we get 
\begin{eqnarray*}
& & |\xi_{0} \eta_{1} - \xi_{1} \eta_{0}|  = |\xi_{0}(\omega \eta_{1} - \xi_{1}) - \xi_{1}(\omega  \eta_{0} - \xi_{0})| \\
& \leq & 8 h \sqrt{\left|3I\, A_{4}\right|}\frac{\pi}{12} \left(\frac{|\xi_{0}|}{|\xi_{1}^{3}|} +  \frac{|\xi_{1}|}{|\xi_{0}^{3}|}\right) \leq\frac{4 \pi}{3}h\sqrt{\left|3I\, A_{4}\right|} \left( \frac{|\xi_{1}|}{|\xi_{0}^{3}|}\right),
\end{eqnarray*}
which leads to
\begin{equation}\label{n2}
 |\xi_{1}| \geq \frac{3}{2 \pi h\, \left| A_{4}\right|^{1/4}}|\xi_{0}|^{3} \geq  \frac{72\sqrt{3}}{2 \pi h^4 (5 \pi)^3 \, \left| A_{4}\right|}|\xi_{-1}|^{9} .
\end{equation}
Note that  $\left|\frac{8 h \sqrt{ \left|3I\, A_{4}\right|}}{\xi^{4}_{-1}}\right| = |z_{-1}| = \left|1 - \left(\frac{\eta_{-1}}{\xi_{-1}}\right)^{4}\right| < 2$ and therefore,
\begin{equation}\label{nn2}
|\xi_{-1}|^{4} > 4h \sqrt{\left|3I\, A_{4}\right|} .
\end{equation}
Thus, when $I > 36.6 h^2$ we have
\begin{equation}\label{982}
|\xi_{1}| > I^{\frac{9}{8}} \frac{72 \sqrt{3}\left( 4\sqrt{3}\right)^{9/4} \left| A_{4}\right|^{1/8}}{2 \pi (5 \pi)^3 h^{7/4}} > 0.39 \,
\frac{I^{\frac{9}{8}} \,\left| A_{4}\right|^{1/8} }{h^{7/4}}.
\end{equation}

Recall that By Lemma \ref{red2}, we can assume that $\left| H(x , y)\right| \geq h^{3} 12\sqrt{3I}$ when looking for pairs of solutions $(x , y)$ with $ |y| \geq \frac{h^{3/4}}{ (3I)^{1/8}}$.
This implies
$$
\left|H(x_{-1} , y_{-1})\right| \geq 12 \frac{h^3\sqrt{3I}}{\left|A_{3}^{2} A_{4} \right|},
$$ 
So by (\ref{c62}),
$$
\left| \xi_{-1}\right|^4 = \frac{H \sqrt{|A_{4}|}}{3} \geq 4h^3\sqrt{\left|3I A_{4}\right|}.
$$
Moreover, one may assume that $h>2$, for the case $h=1$ is being addressed when we are treating the Thue equation.  Under these assumptions, we have 
$$|z_{-1}| = \left| \frac{8h \sqrt{ \left|3I\, A_{4}\right|}}{\xi_{i}^{4}} \right|< 1$$
 and by (\ref{n2}) and Lemma \ref{6.12},
\begin{equation}\label{98h2}
|\xi_{1}| > (4\sqrt{3})^{9/4} I^{\frac{9}{8}}h^{11/4} \, \left| A_{4}\right|^{1/8} \left(\frac{ 3}{2 \pi }\right)^4> 
4 h^{11/4} I^{\frac{9}{8}}\, \left| A_{4}\right|^{1/8}.
\end{equation}
Here the point is that the inequality $ |y| \geq \frac{h^{3/4}}{ (3I)^{1/8}}$ provides us with  a good enough lower bound (\ref{98h2}) for the size of $\xi_{1}$.
Hence, to prove Theorem \ref{main22}, we do not need the assumption $I > 36.6 h^2$. 

 \section{Some Algebraic Numbers} 
 
 Combining the polynomials $A_{r , g}$ and $B_{r , g}$ in Lemma \ref{hyp2} with the resolvent forms, we will consider the complex sequences $\Sigma_{r,g} $ given by
 $$
 \Sigma_{r,g} = \frac{\eta_{2}}{\xi_{2}}A_{r,g}(z_{1}) - \frac{\eta_{1}}{\xi_{1}} B_{r,g}(z_{1})
 $$
 where $z_{1} = 1 - \eta_{1}^{4}/ \xi_{1}^{4}$ .
For any pair of integers $(x , y)$, $ \xi^{4}(x , y)$ and $\eta^{4}(x , y)$  are algebraic integers in  $\mathbb{Q}(\sqrt{A_{0}I/3})$(see Lemma \ref{nr2}).  We have seen that $A_{0} < 0$ and one can assume $A_{3} A_{4} \neq 0$ (see Lemma \ref{A342}). Therefore from (\ref{22}), we have $A_{1} \neq 0$.
Define
$$
\Lambda_{r,g} = \left(9|A_{4}|\right)^{\frac{1-g}{4}}\xi_{1}^{4r + 1- g} \xi_{2} \Sigma_{r,g}.
$$
  We will show that  $\Lambda_{r,g}$ is either an integer in $\mathbb{Q}(\sqrt{\frac{A_{0}I}{3}})$ or a fourth root of such an integer.  If $\Lambda_{r,g} \neq 0$, this provides a lower bound upon 
 $|\Lambda_{r,g}|$.

  \begin{lemma}\label{122}
 For any pair of integer $(s , t)$, we have
   \begin{equation*}
 \frac{\xi( s , t)}{\xi(1 , 0)}  , \frac{\eta( s , t)}{\eta(1 , 0)} \in \mathbb{Q}(\sqrt{A_{0}I/3})[s , t].
 \end{equation*}
  \end{lemma}
  \begin{proof}
By (\ref{22}) and (\ref{IA2}), we have
\begin{equation*}
\alpha\delta - \beta\gamma = \sqrt{A_{3}^{4} - 16 A_{1}A_{4}^{2}A_{3} } = \frac{4A_{3}^{2}}{A_{1}}\sqrt{3IA_{0}} = \frac{12A_{3}^{2}}{A_{1}}\sqrt{\frac{IA_{0}}{3}} .
\end{equation*}
Since
\begin{eqnarray*}
2A_{1}A_{4}x^{2} + A_{3}^{2}xy + 2A_{3}A_{4}y^{2} & = &(\alpha x + \beta y) (\gamma x + \delta y) \\
 & = & \sqrt{12 A_{3}^2 \sqrt{|A_{4}|}}\xi(x , y) \eta(x , y) ,
\end{eqnarray*}
we conclude that $\alpha \gamma$ , $\beta \delta$ , $\alpha\delta+ \beta\gamma \in \mathbb{Z}$. Thus, for integral pair $(s , t)$, we obtain
  \begin{equation*}
 \frac{\xi( s , t)}{\xi(1 , 0)}  , \frac{\eta( s , t)}{\eta(1 , 0)} \in \mathbb{Q}(\sqrt{A_{0}I/3})[s , t].
 \end{equation*}
 \end{proof}

 \begin{lemma}\label{ai2}
 If $(x_{1} , y_{1})$ and $(x_{2}, y_{2})$ are two pairs of rational integers then
 $$
 \sqrt{3|A_{4}|^{1/2}}\xi(x_{1} , y_{1}) \eta(x_{2} , y_{2}) , 
 $$
  $$
  \xi(x_{1} , y_{1})^{3} \xi(x_{2} , y_{2}) 
 $$
  and
  $$\eta(x_{1} , y_{1})^{3} \eta(x_{2} , y_{2})$$
   are integers in $\mathbb{Q}(\sqrt{A_{0}I/3})$.
  \end{lemma} 
   \begin{proof}
   For any pair of integers $(x , y)$, Lemma \ref{122} implies that 
$$\frac{\xi(x , y)}{\xi(1 , 0)}  \in \mathbb{Q}(\sqrt{A_{0}I/3}).$$
Thus, $$\frac{\xi(x_{1} , y_{1})}{\xi(x_{2} , y_{2})} \in \mathbb{Q}(\sqrt{A_{0}I/3}).$$ 
Since $$\sqrt{3|A_{4}|^{1/2}}\, \xi(x_{2} , y_{2}) \eta(x_{2} , y_{2}) = \frac{\omega(x , y)}{2|A_{3}|} \in \mathbb{Q},$$ 
the algebraic integer $\sqrt{3|A_{4}|^{1/2}}\xi(x_{1} , y_{1}) \eta(x_{2} , y_{2})$ belongs to $\mathbb{Q}(\sqrt{A_{0}I/3})$.

\noindent  Let $\xi(x , y) = \epsilon_{1} x + \epsilon_{2} y$. Clearly, $\epsilon_{1}$ and $\epsilon_{2}$ are algebraic integers and so are $\epsilon_{1}^{4}$, $\epsilon_{1}^{3} \epsilon_{2}$,
  $\epsilon_{1}^{2} \epsilon_{2}^{2}$, 
$\epsilon_{1} \epsilon_{2}^{3}$ and $\epsilon_{2}^{4}$. Since $\xi^{4}$ is an integer in $\mathbb{Q}(\sqrt{A_{0}I/3})$, we conclude that   $\epsilon_{1}^{4}$, $\epsilon_{1}^{3} \epsilon_{2}$,
  $\epsilon_{1}^{2} \epsilon_{2}^{2}$, 
$\epsilon_{1} \epsilon_{2}^{3}$ and $\epsilon_{2}^{4}$ are all algebraic integers in   $\mathbb{Q}(\sqrt{A_{0}I/3})$. 

\noindent $\xi(x_{1} , y_{1})^{3} \xi(x_{2} , y_{2})$ is an integer in $\mathbb{Q}(\sqrt{A_{0}I/3})$ , because it can be written as a linear combination with rational integer coefficients in   $\epsilon_{1}^{4}$ , $\epsilon_{1}^{3} \epsilon_{2}$,
  $\epsilon_{1}^{2} \epsilon_{2}^{2}$, 
$\epsilon_{1} \epsilon_{2}^{3}$ and $\epsilon_{2}^{4}$.

\noindent We can similarly show that that $\eta(x_{1} , y_{1})^{3} \eta(x_{2} , y_{1})$ is also an integer in $\mathbb{Q}(\sqrt{A_{0}I/3})$.
\end{proof}

  For every polynomial $P(z) = a_{n}z^{n} + a_{n-1}z^{n-1} + \ldots  + a_{1}z + a_{0}$, we define 
$$
P^{*}(x , y) = x^{n} P(y/x) =  a_{0}x^{n} + a_{1}x^{n-1}y + \ldots + a_{n-1}x y^{n-1} + a_{n}y^{n} .
$$
Let $A_{r,g}$ and $B_{r,g}$ be as in (\ref{AB2}) and
 $$
 C_{r,g}(z) = A_{r,g}(1 - z)  , \  D_{r , g}(z) = B_{r,g}(1 - z), 
 $$
where $A_{r , g}$ and $B_{r , g}$ are the polynomials in Lemma \ref{hyp2}. 
 For $z \neq 0$, we have $D_{r , 0}(z) = z^{r} C_{r,0}(z^{-1})$, hence

\begin{eqnarray}\label{star2} \nonumber 
A^{*}_{r , 0}(z , z - \bar{z}) 
& = & z^{r} A_{r , 0} (1 - \frac{\bar{z}}{z}) = z^{r} C_{r,0}(\frac{\bar{z}}{z}) \\   
& = & \bar{z}^{r} D_{r,0}(\frac{z}{\bar{z}}) = \bar{z}^{r} B_{r,0}(1 - \frac{z}{\bar{z}}) \\  \nonumber
& = & B^{*}_{r , 0} (\bar{z} , \bar{z} - z) = \bar{B}^{*}_{r,0}(z , z -\bar{z} ) .
\end{eqnarray}

\begin{lemma}\label{242}
For any pair of integers $(x , y)$, 
$$
A^{*}_{r,g}(\xi^{4}(x , y), \xi^{4}(x , y) -\eta^{4}(x , y))$$
  and
   $$B^{*}_{r,g}(\xi^{4}(x , y), \xi^{4}(x , y) -\eta^{4}(x , y) )$$
    are algebraic integers in 
$\mathbb{Q}(\sqrt{A_{0}I/3})$.
\end{lemma}

\begin{proof}
It is clear that  $$A^{*}_{r,g}(\xi^{4}(x , y), \xi^{4}(x , y) -\eta^{4}(x , y))$$  and $$B^{*}_{r,g}(\xi^{4}(x , y), \xi^{4}(x , y) -\eta^{4}(x , y) )$$ belong to  $\mathbb{Q}(\sqrt{A_{0}I/3})$. So we need only show that they are algebraic integers. This follows immediately from Lemma 4.1 of \cite{Chu2} since  
$$
\xi^{4}(x , y) -\eta^{4}(x , y) =  8 h \sqrt{3I A_{4}}F(x , y).
$$ 
\end{proof}

We now proceed to show that for any $r \in \mathbb{Z}$,  $\Lambda_{r,0}$ and $\Lambda_{r , 1}^{4}$ are integers in $\mathbb{Q}(\sqrt{A_{0}I/3})$.
\begin{eqnarray*}
 \Lambda_{r,g} & = &\left(9|A_{4}|\right)^{\frac{1-g}{4}} \xi_{1}^{4r} \xi_{1}^{1 - g}\xi_{2} \Sigma_{r,g}\\ \nonumber
 & = & \left(9|A_{4}|\right)^{\frac{1-g}{4}}\left(\xi_{1}^{1 -g} \eta_{2}A^{*}_{r,g}(\xi_{1}^{4} , \xi_{1}^{4} - \eta_{1}^{4}) - \xi_{1}^{-g}\xi_{2} \eta_{1} B^{*}_{r,g}(\xi_{1}^{4} , \xi_{1}^{4} - \eta_{1}^{4})\right).
\end{eqnarray*}
For $g=0$, we have 
$$
\Lambda_{r,0} = \left(9|A_{4}|\right)^{\frac{1}{4}}\left(\xi_{1} \eta_{2}A^{*}_{r,0}(\xi_{1}^{4} , \xi_{1}^{4} - \eta_{1}^{4}) - \xi_{2} \eta_{1} B^{*}_{r,0}(\xi_{1}^{4} , \xi_{1}^{4} - \eta_{1}^{4})\right)
$$
By Lemma \ref{ai2}, $\left(9|A_{4}|\right)^{\frac{1}{4}} \left(\xi_{1} \eta_{2}\right)$ and $\left(9|A_{4}|\right)^{\frac{1}{4}} \left(\xi_{2} \eta_{1}\right)$ are integers in $\mathbb{Q}(\sqrt{A_{0}I/3})$. They are also complex conjugates.
From (\ref{star2}), Lemma \ref{242} and the characterization of algebraic integers in quadratic number fields, we conclude that $\Lambda_{r , 0} \in \mathbb{Z}  \sqrt{A_{0}I/3}$.
 By Lemma \ref{ai2} and Lemma \ref{242}, $\Lambda_{r , 1}^{4}$ is an algebraic integer in $\mathbb{Q}(\sqrt{A_{0}I/3})$. Next we will show that $\Lambda^{4}_{r , 1}$ is not an integer when $\Sigma_{r , 1}$ is nonzero.

 Suppose $\Lambda_{r,1}^{4} \in \mathbb{Z}$. Then we have for some  $\rho  \in \{ \pm1 , \pm i \}$, that 
 $\rho \Lambda_{r,1} =  \bar{\Lambda}_{r,1}$. Hence by the definition of $\Lambda_{r,1}$ and since $\xi_{i}$ and $\eta_{i}$ are complex conjugates,
 \begin{eqnarray}\nonumber
 \rho \Sigma_{r , 1} &= &\xi_{1}^{-4r} \xi_{2}^{-1}  \bar{\Lambda}_{r,1}  \\ \nonumber
  & =& \xi_{1}^{-4r} \xi_{2}^{-1} \eta_{1}^{4r}\eta_{2} \left( \frac{\xi_{2}}{\eta_{2}} A_{r,1} \left(1 - \frac{\xi_{1}^{4}}{\eta_{1}^{4}} \right) - \frac{\xi_{1}}{\eta_{1}}B_{r , 1} \left( 1 - \frac{\xi_{1}^{4}}{\eta_{1}^{4}} \right) \right)\\  \nonumber
  & =& \xi_{1}^{-4r} \xi_{2}^{-1} \eta_{1}^{4r}\eta_{2} \left( \frac{\xi_{2}}{\eta_{2}} A_{r,1} \left(1 - \frac{\xi_{1}^{4}}{\eta_{1}^{4}} \right) - \frac{\xi_{1}}{\eta_{1}}B_{r , 1} \left( 1 - \frac{\xi_{1}^{4}}{\eta_{1}^{4}} \right) \right)\\  \nonumber
  &=&  \frac{\eta_{1}^{4r}}{\xi_{1}^{4r}} \left(A_{r,1}\left( 1- \frac{\xi_{1}^{4}}{\eta_{1}^{4}}\right) - \frac{\xi_{1} \eta_{2}}{\xi_{2}\eta_{1}} B_{r ,1} \left( 1- \frac{\xi_{1}^{4}}{\eta_{1}^{4}}\right) \right) .
    \end{eqnarray}
 This, together with Lemmas \ref{ai2} and  \ref{242}, implies that 
\begin{equation}\label{ronf}
\rho\Sigma_{r,1} \in \mathbb{Q}(\sqrt{A_{0}I/3}).
\end{equation} 
We have, by definition,
$$
\Sigma_{r,g} = \frac{\eta_{2}}{\xi_{2}}A_{r,g}(z_{1}) - \frac{\eta_{1}}{\xi_{1}} B_{r,g}(z_{1}) = \frac{\eta}{\xi} \left[\frac{\eta_{2}/\eta}{\xi_{2}/\xi}A_{r,g}(z_{1}) - \frac{\eta_{1}/\eta}{\xi_{1}/\xi} B_{r,g}(z_{1})\right],
$$
where $\eta = \eta(1 , 0)$ and $\xi = \xi (1 , 0)$.
 By  Lemmas \ref{122} and \ref{242},   $$ \frac{\eta_{2}/\eta}{\xi_{2}/\xi}A_{r,g}(z_{1}) - \frac{\eta_{1}/\eta}{\xi_{1}/\xi} B_{r,g}(z_{1}) \in \mathbb{Q}(\sqrt{A_{0}I/3}).$$
  Hence
\begin{equation}\label{fff}
\mathfrak{f} = \mathbb{Q}(\sqrt{A_{0}I/3} , \rho\Sigma_{r,g})  = \mathbb{Q}(\sqrt{A_{0}I/3} , \rho\frac{\xi}{\eta}).
\end{equation}
If we choose complex number $X$ so that $\xi (X , 1) = \eta(X , 1)$ then by Lemma \ref{122}, $X  \in  \mathfrak{f}$. We have $F (X , 1) = \frac{1}{8\sqrt{3IA_{4}}}(\xi^{4}( X , 1) - \eta^{4}(X , 1) )= 0$. Since we have assumed that $F$ is irreducible,  $X$ has degree $4$ over $\mathbb{Q}$.
  But from (\ref{ronf}) and the definition of number field $\mathfrak{f}$ in (\ref{fff}),
$$
X  \in \mathfrak{f} = \mathbb{Q}(\sqrt{A_{0}I/3}).
$$
This contradicts the fact that $X$ has degree $4$ over $\mathbb{Q}$.  We conclude that $\Lambda_{r,1}$ can not be a rational integer. 

From the well-known characterization of algebraic integers in quadratic fields, we may therefore conclude that,
If $\Lambda_{r,g} \neq 0$,  then  for $g \in \{ 0 , 1\}$
  \begin{equation}\label{ub2}
 |\Lambda_{r,g}| \geq 2^{\frac{-g}{4}}(-A_{0}I/3)^{\frac{1}{2} - \frac{3g}{8}}.
 \end{equation}
 
 \section{Approximating Polynomials}
 
 In order to apply (\ref{lb2}), we must make sure that $\Lambda_{r , g}$ or equivalently $\Sigma_{r , g}$ does not vanish. First we will show that for small $r$, $\Sigma_{r , 0} \neq 0$.
  \begin{lemma} \label{nv12}
 Suppose that $(x , y)$ is a pair of solutions to $F(x , y) = \pm 1$ with $I > 135 $ or a pair of solutions to $\left| F(x , y) \right| \leq h$ with $ |y|> \frac{h^{3/4}}{(3I)^{1/8}}$ . For this pair of solutions and $r \in \{1 , 2, 3, 4, 5\}$, we have
   $$
   \Sigma_{r,0} \neq 0.
   $$
   \end{lemma}
   \begin{proof}
   Let $r \in \{1 , 2,  3, 4, 5 \}$. Suppose that $\Sigma_{r , 0} = 0$. From (\ref{ABF2}), we can find for each
  $r$, a polynomial $F_{r}(z) \in \mathbb{Q}[z]$, satisfying 
 $$
 A_{r,0}(z)^{4} - (1 -z)B_{r,0}^{4} = z^{2r+1} F_{r}(z).
 $$
 In fact, using Maple, we have
 $$
 A_{1}(z) = 4 A_{1, 0}(z) = 8 - 5z , 
 $$
 $$
B_{1}(z) = 4 B_{1 , 0}(z) = 8 - 3z , 
$$
$$
F_{1}(z) = 320 - 320z + 81z^{2} ,
$$
 $$
 A_{2}(z) = \frac{32}{3} A_{2,0}(z) = 64 -72z + 15 z^{2}, 
 $$
  $$
  B_{2}(z)= \frac{32}{3} B_{2,0}(z) = 64 -56z + 7 z^{2},
  $$ 
 $$F_{2}(z) = 86016 - 172032z + 114624z^{2} - 28608z^{3} + 2401z^{4},
 $$
 $$
 A_{3}(z)= 128A_{3,0}(z)= 2560 - 4160z + 1872z^{2} -195z^{3},
  $$ 
   $$
  B_{3}(z)= 128B_{3,0}(z)=2560 - 3520z + 1232z^{2} - 77z^{3}, 
  $$ 
  \begin{eqnarray*}
  F_{3}(z) = & & 14057472000 - 42172416000z \\\nonumber
& & + 48483635200z^{2} - 26679910400z^{3} \\ \nonumber
& &+ 7150266240z^{4}
  - 839047040z^{5} \\\nonumber
& & + 35153041z^{6},  
\end{eqnarray*}
 \begin{eqnarray*}
  A_{4}(z) &= & \frac{2048}{5}A_{4,0}(z) \\
& = & 28672 - 60928z + 42432z^{2} - 10608z^{3} + 663z^{4},
 \end{eqnarray*}
  \begin{eqnarray*}
  B_{4}(z) & = & \frac{2048}{5}B_{4,0}(z)\\
& =&  28672 - 53760z + 31680z^{2} -6160z^{3} +  231z^{4},
  \end{eqnarray*}
  \begin{eqnarray*}
  F_{4}(z) = & &13989396348928- 55957585395712z \\\nonumber
 & & + 91916125077504z^{2} - 79896826347520z^{3} \\\nonumber
  & & +39463764078592z^{4} -11050000539648z^{5}\\\nonumber
 & &   + 1648475542656z^{6}   - 113348764800z^{7}  \\\nonumber
& &   + 2847396321z^{8},
\end{eqnarray*}
  \begin{eqnarray*}
   A_{5}(z) & = & \frac{8192}{21}A_{5,0}(z) \\ \nonumber
& = &98304 - 258048z + 243712z^{2} \\\nonumber
& & - 99008z^{3} + 15912z^{4} - 663z^{5},
  \end{eqnarray*}
  \begin{eqnarray*}
  B_{5}(z) &= &\frac{8192}{21}B_{5,0}(z)\\
& = &98304 - 233472z + 194560z^{2}\\ \nonumber
& & - 66880z^{3} + 8360z^{4} -209z^{5}.
  \end{eqnarray*}
  and
\begin{eqnarray*}
 F_{5}(z) = & &121733331812352  - 608666659061760z  \\\nonumber
 & & + 1301756554248192z^{2}-1555026262622208z^{3}\\\nonumber
& & +1136607561252864z^{4} -523630732640256z^{5} \\ \nonumber
 & & + 151029162176512z^{6}  -26204424888320z^{7} \\ \nonumber
& & +2515441608384z^{8} - 113971885760z^{9} \\ \nonumber
& &   +1908029761z^{10}.
 \end{eqnarray*}
 We also define $A_{r}^{*}$ and $B_{r}^{*}$ via
 $$
 A_{r}^{*}(x , y) = x^{r}A_{r}(y/x) ,
 $$
and
 $$
 B_{r}^{*}(x , y) = x^{r}B_{r}(y/x). 
 $$
 Since $\Sigma_{r , 0} $ is assumed to be zero,
 $$
 \frac{\eta_{2}^{4}}{\xi_{2}^{4}} = \frac{\eta_{1}^{4} (B_{r}^{*}(\xi_{1}^{4}, \xi_{1}^{4} -\eta_{1}^{4}))^{4}} {\xi_{1}^{4} (A_{r}^{*}(\xi_{1}^{4} ,  \xi_{1}^{4} - \eta_{1}^{4}))^{4}} .
 $$
Let $\mathfrak{I}_{r}$ be the integral ideal in $\mathbb{Q}(\sqrt{IA_{0}/3})$ generated by  $\xi_{1}^{4} (A^{*}(\xi_{1}^{4} ,  \xi_{1}^{4} - \eta_{1}^{4}))^{4}$ and $\eta_{1}^{4} (B^{*}(\xi_{1}^{4} ,  \xi_{1}^{4} - \eta_{1}^{4}))^{4}$ and $N(\mathfrak{I}_{r})$ be the absolute norm of $\mathfrak{I}_{r}$. Since the ideal generated by $\xi_{1}^{4} (A_{r}^{*}(\xi_{1}^{4} ,  \xi_{1}^{4} - \eta_{1}^{4}))^{4} - 
\eta_{1}^{4}( B_{r}^{*}(\xi_{1}^{4} ,  \xi_{1}^{4} - \eta_{1}^{4}))^{4}$ divides  $(\xi_{2}^{4} -\eta_{2}^{4}).  \mathfrak{I}_{r}$, we obtain 
\begin{eqnarray*}
& & |\xi_{1}|^{4(4r+1)} | A_{r}^{4}(z_{1}) - (1 - z_{1}) B_{r}^{4}(z_{1})|  \\ \nonumber
& = &  |\xi_{1}^{4} (A_{r}^{*}(\xi_{1}^{4} ,  \xi_{1}^{4} - \eta_{1}^{4}))^{4} - \eta_{1}^{4} (B_{r}^{*}(\xi_{1}^{4} ,  \xi_{1}^{4} - \eta_{1}^{4}))^{4}.
\end{eqnarray*}
Since $\mathfrak{I}_{r}$ is an imaginary quadratic field, by (\ref{a'2}), we get
$$|\xi_{1}|^{4(4r+1)} | A_{r}^{4}(z_{1}) - (1 - z_{1}) B_{r}^{4}(z_{1})| \leq  N(\mathfrak{I}_{r})^{1/2} |\xi_{2}^{4} - \eta_{2}^{4}| $$
By  (\ref{ABF2}),
$$A_{r}^{4}(z_{1}) -  (1 - z_{1}) B_{r}^{4}(z_{1}) =  z_{1}^{2r+1} F_{r}(z_{1}) ,$$
and so we conclude
$$ 
|z_{1}|^{2r+1} |F_{r}(z_{1})| \leq  N(\mathfrak{I}_{r})^{1/2} |\xi_{2}^{4} - \eta_{2}^{4}| |\xi_{1}|^{-4(4r+1)} ;
$$
i.e.
 $$
 1 \leq  \frac{ N(\mathfrak{I}_{r})^{1/2} |\xi_{2}^{4} - \eta_{2}^{4}| |\xi_{1}|^{-4(4r+1)}}{ |z_{1}|^{2r+1} |F_{r}(z_{1})| }.
 $$
Since $\xi_{1}^{4} = (\xi_{1}^{4} - \eta_{1}^{4}) (1 - \frac{\eta_{1}^{4}}{\xi_{1}^{4}} )^{-1} =  (\xi_{1}^{4} - \eta_{1}^{4}) z_{1}^{-1}$we obtain
$$
1 \leq  \frac{ N(\mathfrak{I}_{r})^{1/2} |\xi_{2}^{4} - \eta_{2}^{4}||\xi_{1}^{4} - \eta_{1}^{4}|^{-4r-1} |z_{1}|^{2r}}{|F_{r}(z_{1})| }.
$$
Noting that $|z_{1}| = \left |\xi_{1}^{-4}\right|\left|\xi_{1}^{4} - \eta_{1}^{4}\right|$ and $\left|\xi_{i}^{4} - \eta_{i}^{4}\right| = \left|8h\sqrt{3IA_{4}}F(x , y)\right|$, we obtain for $r \in \{ 1 , 2 , 3 , 4 , 5\}$,
\begin{equation}\label{442}
|\xi_{1}|^{8r}  \leq  \frac{ (N(\mathfrak{I}_{r})^{1/2} |\xi_{1}^{4} - \eta_{1}^{4}|^{-4r-1})|8h\sqrt{3IA_{4}}|^{2r+1}}{|F_{r}(z_{1})|}.
\end{equation}

To estimate  $N(\mathfrak{I}_{r})^{1/2} $, we choose a finite extension $\mathbf{M}$ of $\mathbb{Q}(\sqrt{A_{0}I/3})$ so that the ideal generated by $\xi_{1}^{4}$ and $\xi_{1}^{4} - \eta_{1}^{4}$ in $\mathbf{M}$ is a principal ideal, with generator $p$, say. We denote the extension of $\mathfrak{I}_{r}$ to $\mathbf{M}$, by $\mathfrak{I}'_{r}$. Let $\mathfrak{r}_{r}$ be the ideal in $\mathbf{M}$ generated by $A_{r}^{*}(u , v)$ and $B_{r}^{*}(u, v)$, where $u = \frac{\xi_{1}^{4}}{p}$ and  $v = \frac{\xi_{1}^{4} - \eta_{1}^{4}}{p}$. 
Since $A_{r}^{*}(x , x- y) = B_{r}^{*}(y , y-x)$,
\begin{eqnarray}\label{ss2}
p^{4r+1}\mathfrak{r}_{r}^{4} B_{r}^{*}(0 , 1)^{4} &\subset &  p^{4r+1}\mathfrak{r}_{r}^{4} (u , B_{r}^{*}(0 , v)^{4}) ( u - v, B_{r}^{*}(0 , v)^{4})  \\ \nonumber
 &\subset &   p^{4r+1}\mathfrak{r}_{r}^{4} (u , B_{r}^{*}(0 , v)^{4}) ( u - v, A_{r}^{*}(v , v)^{4})  \\ \nonumber
 &\subset &   p^{4r+1}\mathfrak{r}_{r}^{4}( u , u - v) (u , B_{r}^{*}(u , v)^{4}) ( u - v, A_{r}^{*}(u , v)^{4})  \\ \nonumber
 &\subset & p^{4r+1} (uA^{*}(u , v)^{4} , (u - v)B_{r}^{*}(u , v)^{4}) = \mathfrak{I}'_{r},
 \end{eqnarray}
where $(m_{1} , \ldots , m_{n})$  denote the ideal in $\mathbf{M}$ generated by $m_{1} , \ldots , m_{n}$. 

We have
$$
A_{1}^{*}( x , y) - B_{1}^{*}(x , y) = -2y. 
$$
Therefore, 
$$
2(v) \subset (A_{1}^{*}( u , v) , B_{1}^{*}(u , v)) \subset  \mathfrak{r}_{1}, 
$$ 
where $(v)$ is the ideal generated by $v$ in $\mathbf{M}$.
Since $B_{1}^{*}(0 , 1) = -3$, it follows from (\ref{ss2}) that 
$$
1296(\xi_{1}^{4} - \eta_{1}^{4})^{5} \subset 1296 p (\xi_{1}^{4} - \eta_{1}^{4})^{4}  = p^{5}16v^{4}B_{1}^{*}(0 , 1)^{4} \subset \mathfrak{I}'_{1}. 
$$
For $r = 2$, we first observe that 
$$
B_{1}^{*}(x , y)A_{2}^{*}( x , y) -A_{1}^{*}(x , y) B_{2}^{*}(x , y) =-10y^{3}
$$
and 
$$
(-32x + 7y)A_{2}^{*}( x , y) -(-32x+15y) B_{2}^{*}(x , y) = 80xy^{2}.
$$
Therefore, by (\ref{ss2}) we have
$$
80(v)^{2}  \subset (-10v^{3} , 80uv^{2})  \subset (A_{2}^{*}( u , v) , B_{2}^{*}(u , v)) \subset  \mathfrak{r}_{2} .
$$ 
Since $B_{2}^{*}(0 , 1) =7$, we have 
$$  
80^{4}\times 7^{4} (\xi_{1}^{4} - \eta_{1}^{4})^{9}  \subset  80^{4}\times 7^{4} p (\xi_{1}^{4} - \eta_{1}^{4})^{8}  = 80^{4}p^{9}v^{8}  B_{2}^{*}(0 , 1)^{4} \subset \mathfrak{I}'_{2}.
$$
When $r = 3$, we have 
$$
B_{2}^{*}(x , y)A_{3}^{*}( x , y) -A_{2}^{*}(x , y) B_{3}^{*}(x , y) =-210y^{5}
$$
\begin{eqnarray*}
& & (1616x^{2}-1078xy+77y^{2})A_{3}^{*}(x , y) \\ \nonumber
&- &(1616x^{2}-1482xy+195y^{2}) B_{3}^{*}(x,y) \\ \nonumber
& = & -16800x^{2}y^{3}.
\end{eqnarray*}
Substituting $77$ for $B_{3}^{*}(0 , 1)$, we conclude  
\begin{eqnarray*} 
& & 16800^{4}\times 77^{4} (\xi_{1}^{4} - \eta_{1}^{4})^{13}  \subset  16800^{4}\times 77^{4} p (\xi_{1}^{4} - \eta_{1}^{4})^{12} \\ \nonumber
&  = & 16800^{4}p^{13}v^{12}  B_{3}^{*}(0 , 1)^{4} \subset \mathfrak{I}'_{3}.
\end{eqnarray*}
For $r = 4$, setting 
\begin{eqnarray*}
G_{4}(x , y)=14178304x^{3}-15889280x^{2}y+4071760xy^{2}-162393y^{3}, \\
H_{4}(x , y) = 14178304x^{3}-19433856x^{2}y+6714864xy^{2}- 466089y^{3},
 \end{eqnarray*}
we may verify that 
$$
B_{3}^{*}(x , y)A_{4}^{*}( x , y) -A_{3}^{*}(x , y) B_{4}^{*}(x , y) =-6006y^{7}
$$
and 
$$
G_{4}(x , y)A_{4}^{*}(x , y) - H_{4}(x , y)B_{4}^{*}(x , y) = -150678528y^{4}x^{3}.
$$
These two identities imply that
$$
150678528 ^{4}\times 231^{4} (\xi_{1}^{4} - \eta_{1}^{4})^{17}  \subset  150678528^{4}\times 231^{4} p (\xi_{1}^{4} - \eta_{1}^{4})^{16}.
$$
Since this latter quantity is equal to $150678528^{4}p^{17}v^{16}  B_{4}^{*}(0 , 1)^{4}$, from (\ref{ss2}) it follows that
$$
150678528 ^{4}\times 231^{4} (\xi_{1}^{4} - \eta_{1}^{4})^{17}  \subset \mathfrak{I}'_{4}.
$$
Finally, for $r = 5$, set
\begin{eqnarray*}
& & G_{5}(x , y) \\ \nonumber
& = & 43706368x^{4}-69346048x^{3}y+32767856x^{2}y^{2} \\ \nonumber
& & -4764782{x}y^{3}+123519y^{4},
\end{eqnarray*}
\begin{eqnarray*}
& & H_{5}(x , y)\\ \nonumber
& = & 43706368x^{4}-80272640x^{3}y+46006896x^{2}y^{2}\\ \nonumber
& & -8845746xy^{3}+391833y^{4} .                                   
\end{eqnarray*} 
Then we have 
$$
B_{4}^{*}(x , y)A_{5}^{*}( x , y) -A_{4}^{*}(x , y) B_{5}^{*}(x , y) =-14586y^{7}
$$
and 
$$
G_{5}(x , y)A_{5}^{*}(x , y) - H_{5}(x , y)B_{5}^{*}(x , y) = - 134424576 y^{5}x^{4}.
$$
These two identities  imply that
$$
134424576^{4}\times 209^{4} (\xi_{1}^{4} - \eta_{1}^{4})^{21}  \subset 134424576^{4}\times 209^{4} p (\xi_{1}^{4} - \eta_{1}^{4})^{20}. 
$$
So by (\ref{ss2}),
$$
134424576^{4}\times 209^{4} (\xi_{1}^{4} - \eta_{1}^{4})^{21}  \subset  134424576 ^{4}p^{21}v^{20}  B_{5}^{*}(0 , 1)^{4} \subset \mathfrak{I}'_{5}.
$$
From the preceding arguments, we are thus able to deduce the following series of inequalities :
$$
N(\mathfrak{I}_{1})^{1/2} |\xi_{1}^{4} - \eta_{1}^{4}|^{-5} \leq 1296,
$$
$$
N(\mathfrak{I}_{2})^{1/2} |\xi_{1}^{4} - \eta_{1}^{4}|^{-9} \leq 560^{4},
$$
$$
N(\mathfrak{I}_{3})^{1/2} |\xi_{1}^{4} - \eta_{1}^{4}|^{-13} \leq (77\times16800)^{4},
$$
$$
N(\mathfrak{I}_{4})^{1/2} |\xi_{1}^{4} - \eta_{1}^{4}|^{-17} \leq (231\times150678528)^{4}
$$
and
 $$
 N(\mathfrak{I}_{5})^{1/2} |\xi_{1}^{4} - \eta_{1}^{4}|^{-21} \leq (134424576 \times 209)^{4}.
 $$
Substituting any of these in (\ref{442}) provides a contradiction to inequality (\ref{982}) when $I > 135$ and a contradition to (\ref{98h2}) when $|y| > \frac{h^{3/4}}{(3I)^{1/8}}$. Note that under both assumptions $I > 135$ and $|y| > \frac{h^{3/4}}{(3I)^{1/8}}$, the function $|z| = \left |\xi^{-4}\right| \left|8h\sqrt{3IA_{4}}F(x , y)\right|$ is small. This makes $\left|F_{r}(z)\right|$ large enough for our contradictions.
\end{proof}

\begin{lemma}\label{nv2} 
If $r \in \mathbb{N}$ and $h \in \{ 0, 1 \}$, then at most one of $\left\{ \Sigma_{r,0}, \Sigma_{r+h,1} \right\}$ can vanish. 
 \end{lemma}
\begin{proof}
  Let $r$ be a positive integer and  $h \in \{ 0 , 1\}$ . Following an argument of Bennett \cite{Ben2}, we define the matrix $\mathbf{M}$:
   \begin{displaymath}
  \mathbf{M} =
  \left( \begin{array}{ccc}
  A_{r,0}(z_{1}) & A_{r+h , 1}(z_{1}) & \frac{\eta_{1}}{\xi_{1}} \\
  A_{r,0}(z_{1}) & A_{r+h , 1}(z_{1}) & \frac{\eta_{1}}{\xi_{1}} \\
 B_{r,0}(z_{1}) & B_{r+h , 1}(z_{1}) & \frac{\eta_{2}}{\xi_{2}} 
 \end{array} \right) .
 \end{displaymath}
  The determinant of $\mathbf{M}$ is zero because it has two identical rows. Expanding along the first row, we get
  \begin{eqnarray*}
0 & = &A_{r, 0}(z_{1}) \Sigma_{r+h , 1} -   A_{r+h, 1}(z_{1}) \Sigma_{r , 0} \\ \nonumber
& & + \frac{\eta_{2}}{\xi_{2}}( A_{r, 0}(z_{1}) B_{r+h ,1}(z_{1})  - A_{r+h, 1}(z_{1}) B_{r ,0}(z_{1}) ).
\end{eqnarray*}
 If $\Sigma _{r,0} = 0$ and  $\Sigma _{r + h,1} = 0$  then $A_{r, 0}(z_{1}) B_{r+h ,1}(z_{1})  - A_{r+h, 1}(z_{1}) B_{r ,0}(z_{1}) = 0$ which contradicts part (iii) of Lemma \ref{hyp2}.
 \end{proof}

  \section{An Auxiliary Lemma} 
 We now combine the upper bound for $\Lambda_{r,g}$ obtained in (\ref{ub2}) with the lower bounds from  Lemma \ref{hyp2} to prove the following lemma. 
 
 \begin{lemma}\label{c2}
  If $\Sigma_{r,g} \neq 0$, then
  $$c_{1}(r,g)  |\xi_{1}|^{4r+1-g}|\xi_{2}|^{-3} + c_{2}(r,g) |\xi_{1}|^{-4r-3(1-g)}|\xi_{2}| > 1 ,$$ 
 where we may take
 $$
 c_{1}(1 , 0) = 4\pi  h\left(\frac{3\left|A_{4}\right|^{3/2}}{\left|A_{0}\right|}\right) ^{1/2} 
 $$
 and
 \begin{eqnarray*}
 c_{2}(1 , 0) = 
 27\,h^3 \left(\frac{3\left|A_{4}\right|^{1/2}}{\left|A_{0}\right|}\right) ^{1/2} (9\sqrt{3I
\left| A_{4}\right|})^{2}\frac{5}{128}
 \end{eqnarray*}
 and for $(r , g) \neq (1, 0)$,
 $$
 c_{1}(r , g) = 2 \sqrt{\pi} \, h \left(\frac{3\left|A_{4}\right|^{3/2}}{\left|A_{0}\right|}\right) ^{1/2} \left(\frac{3|A_{4}|}{|A_{0}|^{3/2}}\right) ^{-g/4}\frac{4^{r}} {\sqrt{ r}}
 $$
 and
 \begin{eqnarray*}
&& c_{2}(r , g) = \\
&&27\, h^{2r + 1 - g} \left(\frac{3\left|A_{4}\right|^{1/2}}{\left|A_{0}\right|}\right) ^{1/2} \left(\frac{3|A_{4}|}{|A_{0}|^{3/2}}\right)^{-g/4}(9\sqrt{3I \left| A_{4}\right|})^{2r-g}\frac{\sqrt{2}}{\sqrt{r}\pi4^{r}}.
 \end{eqnarray*}
  \end{lemma}
  
  \begin{proof}
  By the definition of $\Lambda_{r , g}$ and  (\ref{ABF2}), we can write
 $$|\Lambda_{r,g}|= (9|A_{4}|)^{(1-g)/4} |\xi_{1}|^{4r+1-g}  |\xi_{2}|    \left|(\frac{\eta_{2}}{\xi_{2}} - \omega)A_{r,g}(z_{1}) + \omega z_{1}^{2r+1-g}F_{r,g}(z_{1})\right|. $$
  Since $|1 - z_{1}| = 1$ , $|z_{1}| \leq 1$ and $|z_{i}| = \frac{8 h \sqrt{3I}}{|\xi^{4}_{i}|}$,  by (\ref{F2}),  (\ref{A2}) and inequality (\ref{Gap22}), we have
   \begin{equation}\label{cc2}
   |\Lambda_{r,g}| \leq   (9|A_{4}|)^{(1-g)/4}   |\xi_{1}|^{4r+1-g}   |\xi_{2}| \,\mathfrak{L},\end{equation}
where $\mathfrak{L}$ is equal to 
\begin{eqnarray*}
 {2r - g \choose r}   \frac{2\pi h \sqrt{3I  \left| A_{4}\right|}}{3|\xi^{4}_{2}|}  +  \frac{{r-g+1/4 \choose r+1-g} {r - 1/4 \choose r}}{{2r+1-g \choose r}} 
\left( \frac{9 h \sqrt{3I  \left| A_{4}\right|}}{|\xi^{4}_{1}|}\right)^{2r+1-g}  . 
   \end{eqnarray*}
     Comparing this with (\ref{ub2}), we obtain
   $$
   c_{1}(r,g)  |\xi_{1}|^{4r+1-g}|\xi_{2}|^{-3} + c_{2}(r,g) |\xi_{1}|^{-4r-3(1-g)}|\xi_{2}| > 1 ,
   $$
   where we may take  $c_{1}$ and $c_{2}$ so that 
    \begin{eqnarray*}
  & & c_{1}(r,g) \geq  \\
& &2 \pi\, h  \left(\frac{3 \left| A_{4}\right|^{3/2} }{\left|A_{0}\right|}\right) ^{1/2} \left(\frac{3|A_{4}|}{\left|A_{0}\right|^{3/2}}\right) ^{-g/4} {2r \choose r}  
 \end{eqnarray*} 
  and
\begin{eqnarray*}
& & c_{2}(r,g) \geq \\
 & & 27\, h^{2r+1-g}\left(\frac{3  \left| A_{4}\right|^{1/2}}{\left|A_{0}\right|}\right) ^{1/2} \left(\frac{3|A_{4}|}{|A_{0}|^{3/2}}\right)^{-g/4}(9\sqrt{3I\left| A_{4}\right|})^{2r-g} \frac{{r-g+1/4 \choose r+1-g} {r - 1/4 \choose r}}{{2r+1-g \choose r}} .
\end{eqnarray*}
Substituting $r = 1$ and $g = 0$, we get the desired values for $c_{1}(1 , 0)$ and $c_{2} (1 , 0)$.
 Let us apply the following version of Stirling's formula (see Theorem (5.44) of \cite{Str2}):
$$\frac{1}{2\sqrt{k}}4^{k} \leq {2k \choose k} < \frac{1}{\sqrt{\pi k}}  4^{k} ,$$
for $k \in \mathbb{N}$. 
This leads to the stated  choice of $c_{1}$ immediately.

To evaluate $c_{2}(r , g)$, we first note that
$$
{2r+1-g \choose r} \geq {2r \choose r} \geq \frac{4^{r}}{2\sqrt{r}}.
$$
Next we will show that
\begin{equation}\label{c112}
{r-g+1/4 \choose r+1-g} {r - 1/4 \choose r} < \frac{1}{\sqrt{2}\pi r} ,
\end{equation}
for $r \in \mathbb{N}$ and $g \in \{ 0 , 1\}$, whence we may conclude that
$$ 
 \frac{{r-g+1/4 \choose r+1-g} {r - 1/4 \choose r}}{{2r+1-g \choose r}} <\frac{\sqrt{2}}{\sqrt{r}\pi4^{r}}.
  $$
  This leads immediately to the stated choice of $c_{2}$.
  It remains to show (\ref{c112}).
  Let us set 
   $$X_{r} = {r - 3/4 \choose r} {r - 1/4 \choose r}  = \frac{y_{r}}{r},$$
   whereby
  $$X_{r+1} = {r +1/4 \choose r+1} {r + 3/4 \choose r+1}   = \left( \frac{r^{2} + r + 2/9}{r^{2 }+ r} \right)\frac{y_{r}}{r+1}.$$
  This implies
  $$ y_{1} = 3/16   \     ,    \    y_{r} = \frac{3}{16} \prod_{k = 1}^{r -1} \frac{k^{2} + k + 3/16}{k^{2} + k}.$$
   Since
   $$\prod_{k = 1}^{\infty} \frac{k^{2} + k + 3/16}{k^{2} + k}  = \frac{16}{3 \Gamma(1/4) \Gamma(3/4)} =\frac{16}{3\sqrt{2}\pi},  $$
  we obtain
    $$X_{r} < \frac{1}{\sqrt{2}\pi r}.$$
    For $r \in \mathbb{N}$, we have
     $${r - 3/4 \choose r} >{r + 1/4 \choose r+1}.$$ 
     So when $ g \in \{0 , 1\}$,
     $$
     {r-g+1/4 \choose r+1-g} {r - 1/4 \choose r} \leq X_{r},
     $$
           which completes the proof.
   \end{proof}

 \section{Proof of the Main Theorems} 

  Let us now assume that there are  $4$ distinct solutions $(x_{i} , y_{i})$ to reduced form 
$$\left|F(x , y)\right| \leq h$$
 related to $\omega$ with $|y_{i}| >  \frac{h^{3/4}}{ (3I)^{1/8}}$, corresponding to $\xi_{-1}$, $\xi_{0}$, $\xi_{1}$ and $\xi_{2}$, where we have ordered these in nondecreasing modulus.  We will deduce a contradiction, implying that at most $3$ such solutions can exist. Then Theorem \ref{main22} will be proven, since there are $4$ choices of $\omega$.

We will show that $|\xi_{2}|$ is arbitrarily large in relation to $|\xi_{1}|$. By  (\ref{982}) and (\ref{98h2}), we know that $|\xi_{1}|$ is large and hence $|\xi_{2}|$ is  arbitrarily large, a contradiction.

 \begin{lemma}\label{fin2}
 Let $F(x , y)$ be the quartic form. Suppose that  $(x_{1} , y_{1})$ and $(x_{2} , y_{2})$  are $2$ pairs of solutions to $\left|F(x , y)\right| \leq h$, both related to  $\omega$, a fixed fourth root of unity. Put $\xi_{j} = \xi(x_{j} , y_{j})$. Assume further that either 
 \flushleft
 \begin{itemize}
 \item[(i)] $F(x , y)$ is the quartic form in Theorem $\ref{main22}$ with
  $$
4 \left|A_{4}\right|^{1/8} h^{11/4} I^{9/8} <  |\xi_{1}| < |\xi_{2}| ,
 $$
 or 
 \item[(ii)] $F(x , y)$ is the quartic form in Theorem $\ref{main2}$ with $I > 135$ and
   \begin{equation}\label{982m}
0.39 \left|A_{4}\right|^{1/8} h^{11/4} I^{9/8} <  |\xi_{1}| < |\xi_{2}|.
 \end{equation}
 \end{itemize}
   Then, for each positive integer $r$, 
 $$
 |\xi_{2}| >\frac{4^{r}\sqrt{r}}{27} \frac{|A_{0}|^{1/8}}{\left(3\left|A_{4}\right|^{1/2}\right)^{1/2}h^{2r+1}}(9\sqrt{3I\left|A_{4}\right|})^{-2r} |\xi_{1}|^{4r + 3}.  
   $$
  \end{lemma}
 
 \begin{proof}
 We will use the upper bound  (\ref{98h2}) for case (i) and the upper bound (\ref{982m}) for case (ii).  Note that (\ref{982m}) is a generalization  for the upper bound (\ref{982}) obtained to treat the equation $|F(x , y)| = 1$. By (\ref{n2}), $|\xi_{2}| \geq \frac{3|\xi_{1}|^{3}}{2 \pi h  \left|A_{4}\right|^{1/4}}$. This implies    
  $$ 
  c_{1}(1 , 0)  |\xi_{1}|^{5} |\xi_{2}|^{-3} \leq 4\, h^4 \pi \times 12^{3/2}\left( \frac{3}{\left|A_{0}\right|}\right)^{1/2} \left|A_{4}\right| \left| \xi_{1} \right|^{-4}
  $$
 Therefore, by (\ref{98h2}) or (\ref{982m}) and from the fact that $\left| A_{4}\right|  < 4 I$, we obtain 
 $$
  c_{1}(1 , 0)  |\xi_{1}|^{5} |\xi_{2}|^{-3} < 0.01  $$
Lemma \ref{nv12} implies that  $\Sigma_{1 , 0} \neq 0$. So we may apply Lemma \ref{c2}  to get
 $$c_{2}(1 , 0) I^{3} |\xi_{1}|^{-7}|\xi_{2}| > 0.99 .$$
  One may now conclude
   $$
   |\xi_{2}| > \frac{0.99}{c_{2}(1 , 0)} |\xi_{1}|^{7} > 0.93\, h^{-3} \left( \frac{3\left|A_{4}\right|^{1/2}}{\left|A_{0}\right|}\right)^{-1/2}(9\sqrt{3I  \left|A_{4}\right|})^{-2}  |\xi'_{1}|^{7}. 
   $$
  This proves the lemma for $r = 1$.   Moreover, we may conclude that
  $$
c_{1}(2 ,0)|\xi_{1}|^{9}|\xi_{2}|^{-3} < \frac{18\, h^{10} \sqrt{\pi}\times 16 \times (5 \times 27)^{3}  \left|A_{4}\right|}{|A_{0}|^{2}127^{3}\sqrt{2}}  \left( 9\sqrt{3I  \left|A_{4}\right|} \right)^{6}  \left| \xi_{1} \right|^{-12}.
$$
Since $\left| A_{4} \right| \leq 4I$ , by (\ref{98h2}) or (\ref{982m}) we have 
$$
c_{1}(2 ,0)|\xi_{1}|^{9}|\xi_{2}|^{-3} < 0.1.
$$
Via Lemmas \ref{c2} and \ref{nv12}, we obtain
$$
|\xi_{2}| > \frac{0.9}{c_{2}(1 , 0)} |\xi_{1}|^{11}.
$$
This leads to the proof of the Lemma for $r = 2$, after substituting the value of $c_{2}(2 , 0)$.
  To complete the proof, we use induction on $r$.
  Suppose that for some $r\geq 2$,
 $$
  |\xi_{2}| >\frac{4^{r}\sqrt{r}}{27} \frac{|A_{0}|^{1/8}}{\left(3\left|A_{4}\right|^{1/2}\right)^{1/2}h^{2r+1}}(9\sqrt{3I \left|A_{4}\right|})^{-2r} |\xi_{1}|^{4r + 3}.
  $$  
 Then
\begin{eqnarray*}
& & c_{1}(r+1 ,0)|\xi_{1}|^{4r+5}|\xi_{2}|^{-3} < \\
& & \frac{18\, \sqrt{\pi} \times 27^{3} \left|A_{4}\right| h^{6r+4}}{|A_{0}|^{7/8}4^{2r-1}\sqrt{ (r + 1)} r\sqrt{r}} \left( 9\sqrt{3I \left|A_{4}\right|} \right)^{6r}  \left| \xi_{1} \right|^{-8r -4}.
\end{eqnarray*}
By (\ref{98h2}) or (\ref{982m}), we have 
$$
c_{1}(r+1 ,0) |\xi_{1}|^{4r+5}|\xi_{2}|^{-3} < 0.1.
$$
 If $\Sigma_{r+1 , 0} \neq 0$, then  by Lemma \ref{c2},
$$ 
c_{2}(r+1 , 0) |\xi_{1}|^{-4(r+1)-3}|\xi_{2}| > 0.9.
$$
Hence,
\begin{eqnarray*}
& & |\xi_{2}| > \frac{0.9}{c_{2}(r +1 , 0)} |\xi_{1}|^{4(r+1) + 3}\\
 & > &   \frac{4^{r+1}\sqrt{r+1}}{27\,h^{2r+3}}\left(\frac{\left|A_{0}\right|}{3\left|A_{4}\right|^{1/2}}\right)^{1/2}(9\sqrt{3I})^{-2r-2} |\xi_{1}|^{4r + 7}.
 \end{eqnarray*}
If, however, $\Sigma_{r+1, 0} = 0$, then  by Lemma \ref{nv2}, both $\Sigma_{r+1,1}$ and $\Sigma_{r+2 ,1}$ are both non-zero and by Lemma \ref{nv12}, we have $r > 5$. Using the induction hypothesis, we get
$$c_{1}(r + 1,1)  |\xi_{1}|^{4r+4}|\xi_{2}|^{-3} < 0.01$$
and thus by Lemma \ref{c2}, (\ref{22}) and (\ref{982}), we conclude
$$c_{2}(r+1,1) |\xi_{1}|^{-4r-4}|\xi_{2}| >0.99 .$$
So, we obtain
$$|\xi_{2}| > \frac{4^{r+1}\sqrt{r+1}|A_{0}|^{1/8}}{27\, h^{2r+2} \left|3A_{4}\right|^{1/4} }\left( 9 \sqrt{3I \left|A_{4}\right|}\right)^{-2r-1} |\xi_{1}|^{4(r+1)}.$$
Consequently,   $c_{1}(r+2 , 1) |\xi_{1}|^{4r+8}|\xi_{2}|^{-3}$ is less than
\begin{equation*}
\frac{2 \, \sqrt{\pi}\times 27 \left(3  \left|A_{4}\right|\right) \left( 9 \sqrt{3I \left|A_{4}\right|}\right)^{6r+3} h^{6r+7}
}{4^{2r+1}(r+1)\sqrt{(r+1)\,(r+2)\,|A_{0}|} } 
|\xi_{1}|^{-8r-4} 
< 0.1.
\end{equation*}
 A final application of Lemma \ref{c2} implies
$$c_{2}(r+2,1)|\xi_{1}|^{-4r-8}|\xi_{2}| >0.9$$
or
$$|\xi_{2}| > \frac{0.9}{c_{2}(r+2 , 1)} |\xi_{1}|^{4r+8} . $$
It follows that
$$
|\xi_{2}| >\frac{\sqrt{r+2}\,  4^{r+2}}{27} \frac{|A_{0}|^{1/8}}{ 3^{1/4}h^{2r+4}}\left(9\sqrt{3I \left|A_{4}\right|}\right)^{-2r-3} |\xi_{1}|^{4(r+ 1) + 4}.
$$
Since $|\xi_{1}| >  4 I^{9/8}h^{11/4} \left|A_{4}\right|^{1/8}$, we conclude that
$$
|\xi_{2}| >\frac{4^{r+1}\sqrt{r+1}}{27} \frac{|A_{0}|^{1/8}}{\left(3\left|A_{4}\right|^{1/2}\right)^{1/2}}(9\sqrt{3I  \left|A_{4}\right|})^{-2r-2} |\xi_{1}|^{4r + 7}.
$$
\end{proof}

\section{Forms With Small Discriminant}\label{FWSD2}
         
 To finish the proof of Theorem \ref{main2}, we need to study the quartic forms $F(x , y)= a_{0}x^{4} + a_{1}x^{3}y + a_{2}x^{2}y^{2} + a_{3}xy^{3} + a_{4}y^{4}$ with $0 < I_{F} \leq 135$ and $A_{0}= 3(8a_{0}a_{2} - 3a_{1}^{2}) < 0$.  
  
We followed  an algorithm of Cremona, in Section 4.6 of  \cite{Cre2}, which gives all inequivalent integer quartics with given invariant $I$ and $J = 0$. 
 Using Magma, we counted the number of solutions to
 $$
 \left| F(x , y)\right| =1,
 $$
  for all reduced quartic forms $F$ with $I_{F} \leq 135$ and $J_{F} = 0$. Regarding $(x , y)$ and $(-x , -y)$ as the same, we didn't find any form $F$ for which there are more than $4$ solutions to $F(x , y) = \pm 1$. 
Our programming was not efficient in the sense that it solves more than one equation from some equivalent classes. While reading the earlier versions of this paper, the  referee has verified these computations in a very efficient way and kindly shared his results with the author. The following table contains all representatives of the complete set of binary forms $F$ with $I_{F} \leq 135$ and $J_{F} = 0$ that split in $\mathbb{R}$.

\bigskip

\begin{tabular}{l @{ \qquad } l}
\hline
$F(x , y)$ & $I_{F}$\\
\hline
$x^4 - x^3y - 6x^2y^2 + xy^3 + y^4$ & $51$\\
$x^4 + 2x^3y - 6x^2y^2 - 2xy^3 + y^4$ & $60$ \\
$x^4  - 12 x^2 y^2 + 16 x y^3 - 4y^4$ & $96$ \\
$x^4 + 8 x^3 y + 6 x^2 y^2 - 4 x y^3 - 2 y^4$ & $108$\\
$x^4 + x^3y - 15 x^2 y^2 + 18 x y^3 - 4 y^4$ & $123$\\
\hline \\
\end{tabular}

\noindent are representatives of the complete set of binary forms $F$ with $I_{F} \leq 135$ and $J_{F} = 0$ that split in $\mathbb{R}$. To solve the Thue equations $F(x , y)= \pm 1$ for forms $F$ in above table, we may also use PARI since all of the binary forms in the table are monic.

\noindent If $$F(x , y) = x^4 -x^3y -6x^2y^2 + xy^3+y^4$$ then  $I_{F} = 51$ and the solutions are:
$$(-1 , 0), (0 , 1), (1 , 2), (-2 , 1).$$
Note that we can write
$$F(x , y) = x^4 -x^3y -6x^2y^2 + xy^3+y^4 = \xi^4(x , y) - \eta^{4}(x , y),$$
so that $\frac{\eta(x , y)}{\xi(x , y)} = \frac{x - iy}{x+iy}$ and we have $\frac{\eta(-1, 0)}{\xi(-1 , 0)} = 1$, $\frac{\eta(0 , 1)}{\xi(0 , 1)} = -1$, $\frac{\eta(1 , 2)}{\xi(1 , 2)} = -\frac{3+4i}{5}$ and $\frac{\eta(-2 , 1)}{\xi(-2 , 1)} = \frac{3+4i}{5}$. This means 
$(-1 , 0)$ is related to $\omega = 1$, $(0 , 1)$  is related to $\omega = - 1$, $(1 , 2)$  is related to $\omega = -i$ and  $(-2 , 1)$  is related to $\omega = i$. Therefore, related to each root of unity there is one pair of solution.

\bigskip

\noindent If $$
F(x , y) = x^4 + 2x^3y - 6x^2y^2 - 2xy^3 + y^4
$$
Then $I_{F} = 60$ and the solutions are $(1 , 0)$ and $(0 , 1)$.

\bigskip

\noindent If $$F(x , y) = x^4 - 12x^2y^2 + 16xy^3 - 4y^4$$ then $I_{F} = 96$ and
$$F(x , y)= 1$$
has $4$ solutions $(5 , 2)$, $(1 , 3)$, $(1 , 1)$, $(1 , 0)$ and the equation
$$F(x , y) = -1$$ has no solution.

\bigskip
  
 \noindent If $$
F(x , y) = x^4 + 8 x^3 y + 6 x^2 y^2 - 4 x y^3 - 2 y^4
$$ 
then $I_{F} = 108$ and the solutions are $(1 , 0)$ and $(-1 , 1)$.

\bigskip

\noindent If $$
F(x , y) = x^4 + x^3y - 15 x^2 y^2 + 18 x y^3 - 4 y^4$$ then $I_{F} = 123$ and the solutions are $(1 , 1)$ and $(1 , 0)$.

\bigskip

\section{Acknowledgments}
 
 The author would like to thank Professor Michael Bennett for his support and insightful comments. The author is  indebted to the anonymous referee for his very careful reading and valuable comments on the earlier version of this paper. The referee's suggestions certainly  improved both presentation and mathematical contents of this manuscript

\end{document}